\documentclass[10pt]{amsart}
\usepackage{geometry}                
\usepackage{graphicx}
\usepackage{amssymb}
\usepackage{epstopdf}
\usepackage{xcolor}

\newcommand{\comment}[1]{}

\renewcommand{\epsilon}{\varepsilon}
\renewcommand{\Re}{\mathrm{Re}\,}
\renewcommand{\Im}{\mathrm{Im}\,}

\theoremstyle{definition}
\newtheorem{defn}{Definition}[section]

\theoremstyle{remark}
\newtheorem{exa}[defn]{Example}

\theoremstyle{plain}
\newtheorem{prop}[defn]{Proposition}
\newtheorem{cor}[defn]{Corollary}
\newtheorem{thm}[defn]{Theorem}
\newtheorem{lem}[defn]{Lemma}

\title{Smooth and Gevrey Microlocal Hypoellipticity for a Class of Hypocomplex Tube Structures}
\author{Nicholas Braun Rodrigues}

\begin{document}
\maketitle

\begin{abstract}
We prove a smooth and Gevrey-$s$ microlocal hypoellipticity result for a system of complex vector fields associated with a real-analytic locally integrable structure of tube type, that is also microlocal hypocomplex. In order to so, we employ the use of a certain partial F.B.I. transform adapted to the locally integrable structure, first introduced by M. S. Baouendi, C. H. Chang and F. Treves, and we prove a microlocal characterization of the smooth and Gevrey-$s$ wave front set in terms of the decay of this partial F.B.I. transform. \end{abstract}

\section{Introduction}

Let $\mathcal{V} \subset \mathbb{C}\mathrm{T} \Omega$ be a real-analytic locally integrable structure of corank $m$ on a $\mathcal{C}^\omega$-smooth manifold $\Omega$ of dimension $n+m$, \textit{i.e.} it's orthogonal bundle $\mathrm{T}^\prime \subset \mathbb{C}\mathrm{T}^\ast \Omega$, defined by the one forms that annihilate the sections of $\mathcal{V}$, is locally generated by the differentials $m$ of real-analytic functions. Locally, the bundle $\mathrm{T}^\prime$ is generated a small open neighborhood of the origin $U$ by the differential of maps of the form
\begin{equation*}
	Z_k(x,t) = x_k + i \phi_k(x,t),
	\eqno{k = 1, \dots, m.}
\end{equation*} 
where the map $\phi: U \longrightarrow \mathbb{R}^m$ is real-analytic and satisfies $\phi(0) = 0$ and $\mathrm{d}_x \phi(0) = 0$. Assuming $U$ small enough, the sections of $\mathcal{V}$ on $U$ are spanned by the following family of complex vector fields,
\begin{equation*}
	\mathrm{L}_j = \frac{\partial}{\partial t_j} - i \sum_{k,\ell = 1}^m \frac{\partial \phi_k}{\partial t_j} \big(Z_x^{-1}\big)_{\ell,k}\frac{\partial}{\partial x_\ell},
	\eqno{j = 1, \dots, n.}
\end{equation*}
A long standing question in the theory of locally integrable structures, also in CR geometry ($\mathcal{V}$ is said to be CR if $\mathcal{V} \cap \overline{V} = \{0\}$), is to characterize when the system of complex vector fields $\mathrm{L}_1, \dots, \mathrm{L}_n$ is (microlocal) hypoelliptic (analytic hypoelliptic, Gevrey-$s$ hypoelliptic) at the origin. We recall that $\mathcal{V}$ is said to be hypoelliptic at the origin, if every distributional solution $u \in \mathcal{D}^\prime(U)$ of the system
\begin{equation*}
	\mathrm{L}_j u  = f_j,
	\eqno{j = 1, \dots n,}
\end{equation*} 
for $f_j \in \mathcal{C}^\infty(U)$, $j = 1, \dots, n$, is actually $\mathcal{C}^\infty$-smooth in some open neighborhood of the origin. Here we point out that in the CR case (embedded), the analytic hypoellipticity of $\mathcal{V}$ is equivalent to holomorphic extension of CR functions (or CR distributions). 

In 1981, F. Treves conjectured that, when $m = 1$, the hypoellipticity of $\mathcal{V}$ should be equivalent to the openness of the map $Z$ near the origin. It is a well-known fact that the openness of $Z$ near the origin is equivalent to the analytic hypoelipticity of $\mathcal{V}$ (it is actually equivalent to the hypocomplexity of $\mathcal{V}$, see \cite{trevesbook}). In 1980, H. M. Maire proved in \cite{maire1980} that this is true for tube structures, \textit{i.e.} when you can choose the map $\phi$ to be independent of the variable $x$. In \cite{castellanos2013},  J. E. Castellanos, P. D. Cordaro and G. Petronilho proved, for tube structures, that the openness of $Z$ also implies the Gevrey-$s$ hypoellipticity of $\mathcal{V}$ at the origin. The conjecture is still open, but recently, N. Braun Rodrigues, P. D. Cordaro and G. Petronilho proved in \cite{braun2023}, that under the additional assumption that $|\mathrm{d} \Im Z| \leq C | \mathrm{d} Z \wedge \mathrm{d}\bar{Z}|$ (which implies a certain Lojasiewicz inequality), the openness of $Z$ implies the hypoellipticity of $\mathcal{V}$ at the origin. 

For $m > 1$, the picture changes drastically for hypoellipticity, but not for analytic hypoellipticity. In 1982, M. S. Baouendi and F. Treves proved in \cite{baouendi1982} that for tube structures, if $\xi_0 \in \mathbb{R}^m \setminus 0$, then the following are equivalent:
\begin{itemize}
	\item[\textit{1.}] the origin is not a local minimum of the function $t \mapsto \phi(t) \cdot \xi_0$;
	\item[\textit{2.}] $(0,\xi_0) \notin \text{WF}_a(u|_{t=0})$, for every $u \in \mathcal{D}^\prime$ such that $\mathrm{L}_j u = 0$, for $j = 1, \dots, n$. 
\end{itemize}
Condition \textit{2.} is called microlocal hypocomplexity. The structure is said hypocomplex at the origin if every distributional solution $u$ of the homogeneous system $\mathrm{L}_j u = 0$, $j = 1, \dots, n$, is the composition of a holomorphic function $F$ with the first integrals $Z$, \textit{i,e,} $u = F \circ Z$ in an open neighborhood of the origin, or in other words, every distributional solution is hypo-analytic at the origin. Furthermore, a distributional solution $u$ is hypo-analytic at the origin if and only if its trace $u|_{t=0}$ is the composition of a holomorphic function with $Z|_{t=0}$. It is worth noticing that $\mathcal{V}$ is analytic hypoelliptic at the origin if and only if every solution (of the homogeneous system) is real-analytic near the origin. Both properties are in general false for smooth and Gevrey-$s$ regularity. 

For $\mathcal{C}^\infty$ regularity, even for tube structures, it is a whole different story. Consider in $\mathbb{R}^4$ the following complex vector fields introduced by H. M. Maire in \cite{maire1980}:
\begin{align*}
	&\mathrm{L}_1 = \partial_{t_1} + i\left( 3 \partial_{x_1} - (4 t_1 t_2  + 3)t_1^2 \partial_{x_2} \right),\\
	&\mathrm{L}_2 = \partial_{t_2} -i t_1^4 \partial_{x_2}.
\end{align*}
In 1981, M. S. Baouendi and F. Treves proved in \cite{baouendi1981} that the system $\{ \mathrm{L}_1, \mathrm{L}_2 \}$ is analytic hypoelliptic but not hypoelliptic. Then, in 2013, J. E. Castellanos, P. D. Cordaro and G. Petronilho proved in \cite{castellanos2013} that the system $\{ \mathrm{L}_1, \mathrm{L}_2 \}$ is not Gevrey-$s$ hypoelliptic for every $s \geq 4$.

Now let us describe the results contained in the present paper. Suppose $\mathcal{V}$  is of tube type, \textit{i.e.} $\phi(x,t) = \phi(t)$, and let $\xi_0 \in \mathbb{S}^{m-1}$. In section \ref{sec:hypoelliptic} we introduce the property $(\star)$ at $\xi_0$ (see definition \ref{def:star}),  which, in a nutshell, means that the origin is not a local minimum of $t \mapsto \phi(t)\cdot \xi_0$, and there exist constants $C>0$, $1/2 \leq \theta < 1$, such that $|\phi(t) \cdot \xi|^\theta \leq C|{}^t\mathrm{d}\phi(t)\xi|$, for $t$ in an open neighborhood of the origin and $\xi$ in an open neighborhood of $\xi_0$. Let $V \subset \mathbb{R}^n$ and $W \subset \mathbb{R}^m$ be two sufficiently small open neighborhoods of the origin, as described in section \ref{sec:preliminaries}. Under this condition we were able to prove the following Theorem: 
\begin{thm}\label{thm:microlocal_hypoelliptic_tube}
Let $\xi_0 \in \mathbb{S}^{m-1}$ be such that the map $\phi$ satisfies condition $(\star)$ at $\xi_0$. Let $s > 1$, and let $u \in \mathcal{C}^1 (W; \mathcal{D}^\prime(V))$ be such that $(0,0,\xi_0,0) \notin \mathrm{WF}_s(\mathrm{L}_j u)$, for $j = 1, \dots, n$. Then $(0,0,\xi_0,0) \notin \mathrm{WF}_s(u)$. Furthermore, if we assume $\theta = 1/2$ on \eqref{eq:lojasiewicz}, then for every $u \in \mathcal{C}^1 (W; \mathcal{D}^\prime(V))$ such that $(0,0,\xi_0,0) \notin \mathrm{WF}(\mathrm{L}_j u)$, for $j = 1, \dots, n$, we have $(0,0,\xi_0,0) \notin \mathrm{WF}(u)$.
\end{thm}
At the end of the paper we discuss a class of examples given by homogeneous polynomials of a fix degree satisfying condition $(\star)$.

In order to prove Theorem \ref{thm:microlocal_hypoelliptic_tube}, we gave a partial characterization of the $\mathcal{C}^\infty$ and Gevrey-$s$ wave front sets in terms of the decay of a partial adapted F.B.I. transform (see section \ref{sec:FBI}):
\begin{thm}\label{thm:microlocal_Gevrey_FBI}
	Let $s > 1$, $u \in \mathcal{C}^1 (W; \mathcal{D}^\prime (V))$ and $\xi_0 \in \mathbb{R}^m \setminus 0$. Then the following are equivalent:
	\begin{enumerate}
		\item[\textit{1.}]  $(0,0,\xi_0,0) \notin \mathrm{WF}_s(u)$;\\
	 	\item[\textit{2.}] For every $\chi \in \mathcal{C}_c^\infty(V)$, with $0\leq \chi \leq 1$ and $\chi \equiv 1$ in some open neighborhood of the origin, there exist $\widetilde{V} \subset V$, $\widetilde{W} \subset W$, open balls centered at the origin, an open cone $\mathcal{C} \subset \mathbb{R}^m \setminus 0$ containing $\xi_0$, and constants $C,\epsilon > 0$ such that
    	\begin{equation}\label{eq:FBI-decay-gevrey}
		    |\mathfrak{F}[\chi u](t; Z(x,t), \zeta)| \leq C e^{-\epsilon |\zeta|^{\frac{1}{s}}}, \qquad \forall t \in \widetilde{W}, x \in \widetilde{V}, \zeta \in \mathbb{R}\mathrm{T}^\prime_{\Sigma_t}\big|_{Z(x,t)}( \mathcal{C}).
	\end{equation}
\end{enumerate}
\end{thm} 
\begin{thm}\label{thm:microlocal_smooth_FBI}
	Let $u \in \mathcal{C}^1 (W; \mathcal{D}^\prime (V))$ and $\xi_0 \in \mathbb{R}^m \setminus 0$. Then the following are equivalent:
	\begin{enumerate}
		\item[\textit{1.}]  $(0,0,\xi_0,0) \notin \mathrm{WF} (u)$;\\
	 	\item[\textit{2.}] For every $\chi \in \mathcal{C}_c^\infty(V)$, with $0\leq \chi \leq 1$ and $\chi \equiv 1$ in some open neighborhood of the origin, there exist $\widetilde{V} \subset V$, $\widetilde{W} \subset W$, open balls centered at the origin, an open cone $\mathcal{C} \subset \mathbb{R}^m \setminus 0$ containing $\xi_0$, such that for every $N > 0$ exists $C_N > 0$ satisfying
    	\begin{equation*}
		    |\mathfrak{F}[\chi u](t; Z(x,t), \zeta)| \leq \frac{C_N}{(1 + |\zeta|)^N}, \qquad \forall t \in \widetilde{W}, x \in \widetilde{V}, \zeta \in \mathbb{R}\mathrm{T}^\prime_{\Sigma_t}\big|_{Z(x,t)}( \mathcal{C}).
	\end{equation*}
\end{enumerate}
\end{thm}
Theorem \ref{thm:microlocal_Gevrey_FBI} can be seen as a microlocal version of Theorem 3.4 of \cite{braun2022}.
This partial F.B.I. was first introduced by M. S. Baouendi, C.H. Chang and F. Treves in \cite{baouendi1982}, and it was widely used in the study of microlocal analytic hypoellipticity for CR structures (or extensions of CR functions on wedges), see for instance \cite{baouendi1985,baouendi1988,baouendi1982,baouendi1984}. It was also used as a technical tool in order to obtain microlocal regularity of traces, see for instance \cite{berhanu2012,hoepfner2008,hoepfner2010,hoepfner2015,asano,barostichi2009,barostichi2011,berhanu2009,braundasilva2022,braundasilva2021,hanges1992,hoepfner2018}. In 2022, the author of the present paper algo gave in \cite{braun2022} a local characterization of Gevrey-$s$ regularity for distributions $u$ such that $\mathrm{L}_j u \in G^s$, for $j = 1, \dots, n$, in terms of the decay of this partial F.B.I. transform. This partial F.B.I. was also studied by R. D. Medrado in his Ph.D. thesis \cite{medrado2016}, where he introduced the notion of hypo-DC (here DC stands for Denjoy-Carlman) structures.

\section{Preliminaries}\label{sec:preliminaries}

%
%

In this section we shall recall some of the main definitions and results to make the reading of the present paper more accessible, and we refer to \cite{trevesbook} and for a more detailed discussion on locally integrable structures and hypo-analytic structures. Let $\Omega$ be a $\mathcal{C}^\infty$-smooth manifold, by a locally integrable structure $\mathcal{V}$ of rank $n$ on $\Omega$ we mean an $n$-dimensional $\mathcal{C}^\infty$-smooth subbundle of $\mathbb{C} \mathrm{T}\Omega$ such that its orthogonal bundle (in the dual sense), $\mathrm{T}^\prime \subset \mathbb{C} \mathrm{T}^\ast \Omega$, is locally generated by the differentials of $m$, $\mathcal{C}^\infty$-smooth, functions. In the case that $\Omega$ is $\mathcal{C}^\omega$-smooth manifold, we say that $\mathcal{V}$ is a real-analytic locally integrable structure if $\mathcal{V}$ is $\mathcal{C}^\omega$-smooth. In the present paper we shall deal only with real-analytic manifolds endowed with a real-analytic locally integrable structures. 

A crucial object in the study of smooth and Gevrey regularity via the F.B.I. transform is the so-called real structure bundle of maximally real submanifolds of $\mathbb{C}^m$, so we shall briefly recall it's definition. Let $\Sigma \subset \mathbb{C}^m$ be a $\mathcal{C}^\infty$-smooth submanifold. We say that $\Sigma$ is maximally real if, for every $p \in \Sigma$, one of the following equivalent conditions holds true:
\begin{itemize}
    \item 
        $\mathbb{C} \mathrm{T}_p \mathbb{C}^m \simeq \mathrm{T}_p^{(0,1)} \mathbb{C}^m \oplus \mathbb{C} \mathrm{T}_p \Sigma$;
    
    \item 
        The pullback map $j^\ast : \mathbb{C} \mathrm{T}_p^\ast \mathbb{C}^m \rightarrow \mathbb{C} \mathrm{T}_p^\ast \Sigma$, where $j$ is the inclusion map $\Sigma \hookrightarrow \mathbb{C}^m$, induces an isomorphism $j^\ast_{1,0} : {\mathrm{T}_{(1,0)}}_p \mathbb{C}^m \to \mathbb{C} \mathrm{T}_p^\ast \Sigma$,
    
    \item 
        The one forms $\mathrm{d} (z_1|_\Sigma), \dots, \mathrm{d} (z_m|_\Sigma)$ are linearly independent at $p$.
\end{itemize}
Here we are using the notation $\mathrm{T}_p^{(0,1)} \mathbb{C}^m$ for the set of the $(1,0)$-complex vector fields at $p$, \textit{i.e.} the set of all holomorphic vector fields at $p$, and ${\mathrm{T}_{(1,0)}}_p \mathbb{C}^m$ for the $(1,0)$-forms at $p$. The real structure bundle of $\Sigma$ is then defined by the image of $\mathrm{T}^\ast\Sigma$ under the isomorphism $( j_{1,0}^\ast )^{-1}$, and it is denoted by $\mathbb{R} \mathrm{T}^\prime_{\Sigma}$.

Let $\Omega$ be an $n+m$-dimensional $\mathcal{C}^\omega$-smooth manifold endowed with a real-analytic locally integrable structure $\mathcal{V}$ of rank $n$, \textit{i.e.} $\text{dim} \mathcal{V}_p = n$, for all $p \in \Omega$, and let $p_0$ be an arbitrary point in $\Omega$. There is a local coordinate system $(U, x_1, \dots, x_m, t_1, \dots, t_n)$, vanishing at $p_0$, and  a real-analytic map $\phi : U \longrightarrow \mathbb{R}^m$ satisfying $\phi(0) = 0$, and $\mathrm{d}_x \phi(0) = 0$, such that the differential of the functions
\begin{equation}\label{eq:defn_Z}
	Z_k(x,t) = x_k + i \phi_k(x, t),
\end{equation} 
for $k = 1, \dots, m$, span $\mathrm{T}^\prime$ on $U$. Shrinking $U$ if necessary, we shall assume that $Z_x(x,t)$ is invertible in $U$. Thus $\mathcal{V}$ is spanned on $U$ by
\begin{equation}\label{eq:defn_L}
    \mathrm{L}_j = \frac{\partial}{\partial t_j} - i \sum_{k=1}^m \frac{\partial \phi_k}{\partial t_j}(x,t) \mathrm{M}_k,
\end{equation}
for $j = 1,\dots, n$, where 
\begin{equation}\label{eq:defn_M}
    \mathrm{M}_k = \sum_{l=1}^m \left({}^tZ_x(x,t)^{-1} \right)_{kl} \frac{\partial}{\partial x_l}, 
\end{equation}
for $k = 1, \dots, m$. These vector fields satisfies the equations:
\begin{equation*}
    \begin{array}{l l}
        \mathrm{L}_j Z_k = 0 & \mathrm{L}_j t_r = \delta_{jr}\\
        \mathrm{M}_l Z_k = \delta_{lk} & \mathrm{M}_l t_r = 0,
    \end{array}
\end{equation*}
so in particular if $h \in \mathcal{C}^1(U)$, then 
\[
	\mathrm{d}h = \sum_{j=1}^n \mathrm{L}_j h \mathrm{d}t_j + \sum_{k=1}^m \mathrm{M}_k \mathrm{d}Z_k.
\]
Since we are in a local coordinate system, we can identify the open set $U \subset \Omega$ with an open set in $\mathbb{R}^{n+m}$, and then we take two open neighborhoods of the origin $V \subset \mathbb{R}^m$, and $W \subset \mathbb{R}^n$, such that $V \times W \subset U$.  Since the complex vector fields $\mathrm{L}_j$s are elliptic on the $t$ variables, if $u \in \mathcal{D}^\prime(V\times W)$ is such that $\mathrm{L}_j u \in \mathcal{C}^{\infty}(V\times W)$, then $u \in \mathcal{C}^\infty (W; \mathcal{D}^\prime (V))$. For every $t \in W$, the set
\begin{equation*}
    \Sigma_t = \{ Z(x,t) \; : \; x \in V \}
\end{equation*}
is a maximally real submanifold of $\mathbb{C}^m$. The real structure bundle of $\Sigma_t$ is then given by, see \cite{trevesbook},  
\begin{equation*}
    \mathbb{R}\mathrm{T}^\prime_{\Sigma_t} = \{ (Z(x,t), {}^tZ_x(x,t)^{-1} \xi) \; : \; x \in V, \xi \in \mathbb{R}^m \}. 
\end{equation*}
Shrinking once again if necessary, we shall assume that for every $t \in W$, every $z,z^\prime \in \Sigma_t$, and every $\zeta \in \mathbb{R}\mathrm{T}^\prime_{\Sigma_t}|_z \cup \mathbb{R}\mathrm{T}^\prime_{\Sigma_t}|_{z^\prime}$ the following holds true:
\begin{equation*}
    \begin{cases}
        |\Im \zeta| < \kappa |\Re \zeta|;\\
        \Re \{ i \zeta \cdot (z - z^\prime) - \langle \zeta \rangle \langle z - z^\prime \rangle^2 \} < - \kappa^\prime |\zeta| |z-z^\prime|^2,
    \end{cases}
\end{equation*}
for some $\kappa^\prime > 0$ and $0 < \kappa < 1$, where $\langle w \rangle^2 = w \cdot w$, for $w \in \mathbb{C}^m$ (this condition is called \emph{well-positionedness}). For every $u \in \mathcal{C}^1 (W; \mathcal{E}^\prime (V))$ we define the adapted partial F.B.I transform of $u$ as
\begin{equation}\label{eq:defn_FBI}
	\mathfrak{F}[u](t;z,\zeta) = \left \langle u(x,t), e^{i\zeta\cdot(z - Z(x,t)) - \langle \zeta \rangle \langle z - Z(x,t)\rangle^2} \Delta(z - Z(x,t),\zeta) \det Z_x(x,t)) \right\rangle_{\mathcal{D}^\prime_x},
\end{equation}
for every $t \in W$, $z \in \mathbb{C}^m$, and $\zeta \in \mathfrak{C}_1 = \{\eta \in \mathbb{C}^m \; : \; |\Im \eta| < |\Re \eta| \}$, where $\Delta(z,\zeta)$ is the Jacobian of the transformation 
\begin{equation*}
	\zeta \mapsto \zeta + iz\langle\zeta\rangle.
\end{equation*}
The adapted partial F.B.I. transform is holomorphic with respect to $z$ and $\zeta$, $\mathcal{C}^\infty$-smooth with respect to $t$, and it satisfies
\begin{equation}\label{eq:del_t_FBI}
	\begin{cases}
	\frac{\partial}{\partial t_j}\mathfrak{F}[u](t,z,\zeta) = \mathfrak{F}[\mathrm{L}_j u](t,z,\zeta)\\
	\frac{\partial}{\partial z_k}\mathfrak{F}[u](t,z,\zeta) = \mathfrak{F}[\mathrm{M}_k u](t,z,\zeta)
	\end{cases}
\end{equation}
for $j  = 1, \dots, n$, and $k = 1, \dots, m$. We can extend $Z(x,t)$ in the $x$ variable to the whole $\mathbb{R}^m$ by multiplying $\phi$ by a $\mathcal{C}_c^\infty$ function, with support contained in $V$ and identically one in some open ball centered in the origin, in such a way that we can change $V$ by $\mathbb{R}^n$ in above discussion. It is clear that the new structure generated by the extended $Z$ is not real-analytic any more, but it agrees with the original one in some open neighbourhood of the origin, that we are still going to denote by $V \times W$. The advantage of doing this extension is the following inversion formula (see \cite{braun2022}, and for an inversion formula modulo hypo-analytic functions see \cite{trevesbook}):
\begin{thm}\label{thm:FBI-inversion-formula}
Let $k \in [0, \infty]$, and $u\in\mathcal{C}^k (W;\mathcal{E}^\prime(V))$. Then
\begin{equation}\label{eq:FBI-inversion-formula}
    u(x,t)=\lim_{\epsilon\to 0^+}\frac{1}{(2\pi^3)^\frac{m}{2}}\iint_{\mathbb{R}\mathrm{T}^\prime_{\Sigma_t}}e^{i\zeta\cdot(Z(x,t)-z^\prime)-\langle\zeta\rangle\langle Z(x,t)-z^\prime\rangle^2-\epsilon\langle\zeta\rangle^2}\mathfrak{F}[u](t;z^\prime,\zeta)\langle\zeta\rangle^\frac{m}{2}\mathrm{d}z^\prime\wedge\mathrm{d}\zeta,
\end{equation}
where the convergence takes place in $\mathcal{C}^k(W;\mathcal{D}^\prime(V))$.
\end{thm}
To emphasize the important role of the real structure bundle on the study of $\mathcal{C}^\infty$ and Gevrey-smoothness and using these adapted F.B.I. transform we shall make a small digression. Suppose that $n = 0$, \textit{i.e.} we have only the $x$ variables, and let $u$ be a compacty supported distribution on $V$. Denoting by $F_u(x,\xi)$ the usual F.B.I. transform of $u$, namely,
\[
	F_u(x,\xi) = \left\langle u(y), e^{i (x - y) \cdot \xi - |\xi|^2 |x-y|^2} \right\rangle_{\mathcal{D}^\prime_y},
\]
we have that there exist constants $C,N > 0$ such that
\[
	|F_u(x,\xi)| \leq C_N(1+|\xi|)^N,
\]
for every $(x,\xi) \in V \times \mathbb{R}^m$. To obtain a similar bound on the adapted F.B.I. transform one must "stay on the real structure bundle", as can be seen in the following two estimates (see Proposition IX.1.1 and Proposition IX.2.1 of \cite{trevesbook}): 
\begin{enumerate}
	\item For every $K\subset \mathbb{C}^m$ compact set, and for every $0<\tilde{\kappa}<1$, there are constants $C,R>0$ such that
	\begin{equation*}
	    	|\mathfrak{F}[u](z,\zeta)|\leq Ce^{R|\zeta|},\quad \forall z\in K,\,\zeta\in\mathfrak{C}_{\tilde{\kappa}},
	\end{equation*}
	where $\mathfrak{C}_{\tilde{\kappa}} = \{ \zeta \in \mathbb{C}^m \; : \; |\Im \zeta| < \tilde{\kappa} |\Re \zeta| \}$.
	\item There exist an integer $k>0$ and a constant $C>0$ such that
	\begin{equation*}
   		 |\mathfrak{F}[u](z,\zeta)|\leq C(1+|\zeta|)^k,\quad\forall (z,\zeta)\in \mathbb{R}\mathrm{T}^\prime_{\Sigma}.
	\end{equation*}
\end{enumerate}
It also worth noticing that if an estimate holds "outside the real structure bundle", then one obtain holomorphic extendability, as seen in the following Theorem (for a proof see the proof of Theorem IX.3.1 of \cite{trevesbook}):
\begin{thm}\label{thm:importance_real_structure_holomorphic}
  Let $\Sigma \subset \mathbb{C}^m$ be a maximally real submanifolds, and assume that $\Sigma$  is well positioned near the origin. Then there exists $ U\subset\Sigma$ an open neighborhood of the origin such that, for every $u\in\mathcal{E}^\prime(U)$, the following are equivalent:
    \begin{enumerate}
        \item[\textit{1.}] There exist $\mathcal{O}\subset\mathbb{C}^m$, an open neighborhood of the origin, $F$ a holomorphic function on $\mathcal{O}$, such that $u|_{\mathcal{O}\cap \Sigma}=F|_{\mathcal{O}\cap \Sigma}$;\\
        \item[\textit{2.}] There exist $\mathcal{O}\subset\mathbb{C}^m$, an open neighborhood of the origin, $0<\tilde{\kappa}<1$, and $C,\epsilon>0$ such that
        
        \begin{equation*}
            \left|\mathfrak{F}[u](z,\zeta)\right|\leq Ce^{-\epsilon|\zeta|},\quad\forall (z,\zeta)\in\mathcal{O}\times\mathfrak{C}_{\tilde{\kappa}}
        \end{equation*}
       \item[\textit{3.}] There exists $\mathcal{O}\subset\mathbb{C}^m$, an open neighborhood of the origin, such that $\mathfrak{F}[u](z,\xi)$ is dominated by an integrable function with respect to $\xi\in\mathbb{R}^m$, uniformly in $z\in\mathcal{O}$.
    \end{enumerate}
    \end{thm}
As mentioned before, this adapted partial F.B.I. transform was widely used on the study holomorphic extendability (or hypo-analyticity) for solutions of locally-integrable structures, \textit{i.e.} distributions $u$ such that, $\mathbb{L} u =0$, and to the study of the regularity of traces. Here we want to point out one more time, that when dealing with holomorphic extendability of solutions, it is enough to prove the holomorphic extendability of its trace, for if $u(x,0) = F(Z(x,0))$, for some holomorphic function $F$, then the function $F(Z(x,t))$ is a solution, because $\bar{\partial} F= 0$, and it agrees with $u$ when $t = 0$. Then by an application of the Baouendi-Treves approximation formula (see \cite{baouendi1982} and \cite{trevesbook}), one can show that $u = F\circ Z$ in some open neighborhood of the origin. This is not true in general for other types of regularity.

In the proof of Theorem \ref{thm:microlocal_Gevrey_FBI}, we shall use the following result proved in \cite{braundasilva2022}, in the case of real-analytic maximally real submanifolds of $\mathbb{C}^m$, \textit{i.e.} the case $n = 0$:
\begin{thm}[N. Braun Rodrigues and A. V. da Silva Jr.]\label{thm:FBI-antonio-nicholas}
Let $u \in \mathcal{D}^\prime(V)$, let $s > 1$, and $\xi_0 \in \mathbb{R}^m \setminus \{0\}$. The following are equivalent:
\begin{enumerate}
	\item[\textit{1.}] $(0,\xi_0) \notin WF_s(u)$;\\
	\item[\textit{2.}] for every $\chi \in \mathcal{C}^\infty(V)$, with $\chi \equiv 1$ in some open neighborhood of the origin, there exist $V^\prime \subset V$ an open neighborhood of the origin, $\mathcal{C} \subset \mathbb{R}^m \setminus \{0\}$ an open cone containing $\xi_0$, and constants $C,\epsilon >0$ such that
	\begin{equation*}
		|\mathfrak{F}[\chi u](z,\zeta)| \leq C e^{-\epsilon|\zeta|^\frac{1}{s}}, \eqno{\forall (z,\zeta) \in \mathbb{R}\mathrm{T}^\prime_{Z(V^\prime)}(\mathcal{C}). }
	\end{equation*}
\end{enumerate}
\end{thm}

%
%
%
%
%
\subsection{Going to a co-rank zero hypo-analytic structure}\label{sec:co-rank_zero}
%
%
In this section we shall associate to $\mathcal{V}$ a hypo-analytic structure of co-rank zero near the origin. For every $(x,t) \in V \times W$ we set
\begin{equation*}
    \Xi(x,t) = (Z(x,t),t).
\end{equation*}
Notice that $\mathrm{d} \Im \Xi$ it is not necessarily zero at the origin. Since we need well-positionednes to apply the adapted F.B.I. transform technique, we shall compose $\Xi$ with a linear transformation in $\mathbb{C}^{n+m}$ so we end with a well-positioned maximally real submanifold of $\mathbb{C}^{n+m}$. Let $H$ be the following matrix
\begin{equation*}
    H = \left(
        \begin{array}{cc}
            \text{Id} & - i \mathrm{d}_t\phi(0)t \\
            0 & \text{Id} 
        \end{array}
        \right).
\end{equation*}
We define $\widetilde{\Xi}(x,t) = H \Xi(x,t) = (\widetilde{Z}(x,t),t) = (x + i\widetilde{\phi}(x,t),t)$, where $\widetilde{\phi}(x,t) = \phi(x,t) - \mathrm{d}_t\phi(0)t$. Since $\mathrm{d} \Im \widetilde{\Xi}(0) = 0$, shrinking $V$ and $W$ if necessary, we can assume that $\widetilde{\Xi}(V \times W) \subset \mathbb{C}^{m+n}$ is a well-positioned maximally real submanifold, and its real-structure bundle's fiber can be described as 
\begin{equation*}
    {}^t\mathrm{d}\widetilde{\Xi}(x,t)^{-1}(\xi,\tau) = (\zeta, \tau - {}^t\widetilde{Z}_t(x,t)\zeta),
\end{equation*}
where $\zeta = {}^tZ_x(x,t)^{-1}\xi$.  We shall denote $\widetilde{\Theta} = {}^t\mathrm{d}\widetilde{\Xi}(x,t)^{-1}(\xi,\tau)$, and $\mathcal{H} = \widetilde{\Xi}(V \times W)$. So let $0<\tilde{\kappa},\tilde{\kappa}^\prime$, $\tilde{\kappa}<1$, be such that, for every $\widetilde{\Xi}, \widetilde{\Xi}^\prime \in \mathcal{H}$, and every $\widetilde{\Theta} \in \mathbb{R}\mathrm{T}^\prime_{\mathcal{H}}|_{\widetilde{\Xi}} \cup \mathbb{R}\mathrm{T}^\prime_{\mathcal{H}}|_{\widetilde{\Xi}^\prime}$ the following holds true:
\begin{equation*}
    \begin{cases}
        |\Im \widetilde{\Theta}| < \tilde{\kappa} |\Re \widetilde{\Theta}|;\\
        \Re \{ i \widetilde{\Theta} \cdot (\widetilde{\Xi} - \widetilde{\Xi}^\prime) - \langle \widetilde{\Theta} \rangle \langle \widetilde{\Xi} - \widetilde{\Xi}^\prime \rangle^2 \} < - \tilde{\kappa}^\prime |\widetilde{\Theta}| |\widetilde{\Xi} - \widetilde{\Xi}^\prime|^2,
    \end{cases}
\end{equation*}
In view of $|\Re \zeta| < \kappa |\Im \zeta|$, we have that
\begin{equation}\label{eq:comparison<zeta>|zeta|}
	\Re \langle \zeta \rangle  \geq \sqrt{\frac{1-\kappa^2}{1+\kappa^2}}|\zeta|, \quad \left|\frac{\zeta}{\langle \zeta \rangle}\right| \leq \sqrt{\frac{1+\kappa^2}{1-\kappa^2}}
\end{equation}
Furthermore, there exists $c_1,C_1> 0$ such that $|{}^{\mathrm{t}}\widetilde{Z}_t(x,t)\zeta| < c_1|\zeta|$, $|\xi| < C_1|\zeta|$. Let $0<c_0<1$, and assume that $|\tau| < c_0|\xi|$. Then
\begin{equation}\label{eq:comparison-Theta}
	|\widetilde{\Theta}| \leq \sqrt{1 + (c_1 +c_0C_1)^2}|\zeta|.
\end{equation}
%
%
%
\section{Gevrey wave-front set characterization via a partial F.B.I. transform}\label{sec:FBI}
%
%
In this section we shall prove Theorems \ref{thm:microlocal_Gevrey_FBI} and \ref{thm:microlocal_smooth_FBI}. Actually, we shall only present the proof of  Theorem \ref{thm:microlocal_Gevrey_FBI} in view of the similarities between the two proofs, and we use the same notation as in section \ref{sec:preliminaries}. 
\begin{proof}[Proof of Theorem \ref{thm:microlocal_Gevrey_FBI}]
\textit{1}. $\Rightarrow$ \textit{2}.:

\vspace{0,4cm}

Let $\Phi: \mathcal{O}_1 \times \mathcal{O}_2 \longrightarrow \mathbb{C}^m$ the holomorphic extension of $\phi$, and for every $\tau \in \mathcal{O}_2$ we define $\mathcal{Z}_\tau(z) = z + i \Phi(z,\tau)$, where $z \in \mathcal{O}_1$.
Since for every $\tau \in \mathcal{O}_2$ the map $\mathcal{Z}_\tau$ is holomorphic, and $\partial_z \mathcal{Z}_0(0) = \text{Id}$, shrinking if necessary $W$, $V$, and $\mathcal{O}_1$, we have that the map $\mathcal{Z}_t$ is a biholomorphism on $\mathcal{O}_1$, for every $t \in W$. 

For simplicity we shall assume that the restriction of $u$ to an open neighborhood of the origin is the boundary value of a single Gevrey-$s$ almost analytic function, \textit{i.e}. there exist $V_1 \subset V$ and $W_1 \subset W$ two open balls centered origin, $\Gamma \subset \mathbb{R}^{m+n} \setminus 0$ an open convex cone with $(\xi_0 ,0) \cdot \Gamma < 0$, and  $f \in \mathcal{C}^\infty(\mathcal{W}_\delta(V_1 \times W_1, \Gamma))$ satisfying
\begin{equation*}
	\begin{cases}
		|f(z,\tau)| \leq \frac{C_1}{|v|^{k_0}} \\
		|(\bar{\partial}_z + \bar{\partial}_\tau)f(z,\tau)| \leq C_2^{k+1} k!^{s-1} |\Im(z,\tau)|^k 
	\end{cases}
	\eqno{
		\begin{array}{l l}
		(z,\tau) \in \mathcal{W}_\delta(V_1 \times W_1, \Gamma) \\ 
		k \in \mathbb{Z}_+
		\end{array},
		}
\end{equation*}
for some constants $\delta, C_1, C_2>0$, and $k_0 \in \mathbb{Z}_+$, such that $u|_{V_1 \times W_1} = \mathrm{bv} (f)$. Note that it is enough to prove \eqref{eq:FBI-decay-gevrey} for some specific cut-off function $\chi$, because if $\Psi$ is another cut-off function such that $\Psi \equiv 1$ in some neighborhood of the origin, then $\chi - \Psi \equiv 0$ in some neighborhood of the origin. So
let $V_4 \Subset V_3 \Subset V_2 \Subset V_1$ be open neighborhoods of the origin, as small as we want, and let $\chi \in \mathcal{C}_c^\infty(V_2)$, with $\chi \equiv 1$ in $V_3$. Let $v \in \Gamma \cap \mathbb{S}^{n+m-1}$  be fixed. Since $u \in \mathcal{C}^1(W; \mathcal{D}^\prime(V))$ then, for every $t \in W_1$, $x \in V$ and $\xi \in \mathbb{R}^m$, writing $\zeta = {}^tZ_x(x,t)^{-1}\xi$, we have that
\begin{align*}
	\mathfrak{F}[\chi u](t; Z(x,t), \zeta)  = \lim_{\lambda \to 0^+} \int & e^{i \zeta \cdot (Z(x,t) - Z(x^\prime,t)) - \langle \zeta \rangle \langle Z(x,t) - Z(x^\prime,t) \rangle^2} f(x^\prime + i\lambda v^\prime, t + i\lambda v^{\prime\prime}) \chi(x^\prime) \cdot \\
	& \cdot \Delta(Z(x,t) - Z(x^\prime,t), \zeta) \mathrm{d}Z(x^\prime,t).
\end{align*}
In view of $v \cdot \xi_0 < 0$, and shrinking if necessary $V_2$, there exist $W_2 \subset W_1$ an open neighborhood of the origin, $\mathcal{C} \subset \mathbb{R}^m \setminus 0$ an open convex cone containing $\xi_0$, and $c^\sharp > 0$ such that 
\begin{equation*}
	v \cdot \Re \zeta \leq - c^\sharp |\zeta| \eqno{\forall (x,t) \in V_2 \times W_2, \zeta \in \mathbb{R}\mathrm{T}^\prime_{\Sigma_t}\big|_{Z(x,t)}( \mathcal{C}).}
\end{equation*}	
We shall assume that $V_2$ and $W_2$ are small enough so
\begin{equation*}
	|Z(x,t) - Z(x^\prime, t)| < \frac{c^\sharp}{8|v^\prime|} \eqno{\forall x,x^\prime \in V_2, \forall t \in W_2.}
\end{equation*}
 For $0 < \lambda < \delta$, $t \in W_2$, and $w \in \mathcal{Z}_t\left(\{x + i\tilde{v}^\prime\,:\, x\in V_1, (\tilde{v}^\prime,0) + \lambda v \in \Gamma_\delta \} \right)$  we define
\begin{equation*}
    F_{t,\lambda}(w) = f(\mathcal{Z}_t^{-1}(w) + i\lambda v^\prime, t + i\lambda v^{\prime \prime}). 
\end{equation*}
Then $F_{t,\lambda} \in \mathcal{C}^\infty\left(\mathcal{Z}_t\left(\{x + i\tilde{v}^\prime\,:\, x\in V_1, (\tilde{v}^\prime,0) + \lambda v \in \Gamma_\delta \} \right) \right)$, for every $0 < \lambda < \delta$, $t \in W_2$. In view of 
\begin{equation*}
	|\Im \mathcal{Z}_t^{-1}(Z(x,t) + i \gamma v^\prime)| \leq \tilde{C} \gamma |v^\prime| \eqno{\forall (x,t) \in V_2\times W_1, 0 \leq \gamma \ll 1,}
\end{equation*}
there exist $C_3 > 0$, $0 < \delta^\prime < \delta$ and $\sigma > 0$, such that, for every $0 \leq \gamma \leq \sigma$, $0 < \lambda < \delta^\prime$, and $(x, t) \in V_2\times W_2$, the following holds true:
\begin{equation*}
    |\bar{\partial}_w F_{t,\lambda}(Z(x,t) + i\gamma v^\prime)| \leq C_3^{k+1} k!^{s-1} (\gamma + \lambda)^k,
    \eqno{\forall k \in \mathbb{Z}_+.}
\end{equation*}
Let  $ 0 < \lambda < \delta^\prime$, $x \in V_4$, $t \in W_2$, and $\zeta \in \mathbb{R}\mathrm{T}^\prime_{\Sigma_t}\big|_{Z(x,t)}( \mathcal{C})$, \textit{i.e.} $\zeta = {}^tZ_x(x,t)^{-1}\xi$ for some $\xi \in \mathcal{C}$, be fixed. By Stokes's Theorem , and writing $z = Z(x,t)$, we have that
\begin{align*}
	& \int e^{i \zeta \cdot (z - Z(x^\prime,t)) - \langle \zeta \rangle \langle z - Z(x^\prime,t) \rangle^2} f(x^\prime + i\lambda v^\prime, t + i\lambda v^{\prime\prime}) \chi(x^\prime) \Delta(z - Z(x^\prime,t), \zeta) \mathrm{d}Z(x^\prime,t) = \\
	& \int e^{i \zeta \cdot (z - Z(x^\prime,t)) - \langle \zeta \rangle \langle z - Z(x^\prime,t) \rangle^2} F_{t,\lambda}(Z(x^\prime,t)) \chi(x^\prime) \Delta(z - Z(x^\prime,t), \zeta) \mathrm{d}Z(x^\prime,t) = \\
	& = \int_{V_2 \setminus V_3} e^{i \zeta \cdot (z - Z(x^\prime,t)) - \langle \zeta \rangle \langle z - Z(x^\prime,t) \rangle^2} F_{t,\lambda}(Z(x^\prime,t)) \chi(x^\prime) \Delta(z - Z(x^\prime,t), \zeta) \mathrm{d}Z(x^\prime,t)\\
	& + \int_{ V_3} e^{i \zeta \cdot (z - Z(x^\prime,t) - i \sigma v^\prime) - \langle \zeta \rangle \langle z - Z(x^\prime,t) - i \sigma v^\prime \rangle^2} F_{t,\lambda}(Z(x^\prime,t) + i \sigma v^\prime) \Delta(z - Z(x^\prime,t) - i \sigma v^\prime, \zeta) \mathrm{d}Z(x^\prime,t)\\
	& - \int_{Z(\partial V_3,t) + i[0,\sigma]v^\prime} e^{i \zeta \cdot (z - z^\prime) - \langle \zeta \rangle \langle z -z^\prime \rangle^2} F_{t,\lambda}(z^\prime) \Delta(z - z^\prime, \zeta)
	\mathrm{d}z^\prime\\
	& - \int_{ Z(V_3,t) + i[0,\sigma]v^\prime} e^{i \zeta \cdot (z - z^\prime) - \langle \zeta \rangle \langle z -z^\prime \rangle^2} \Delta(z - z^\prime, \zeta) \sum_{k=1}^m \bar{\partial}_{w_k}F_{t,\lambda}(z^\prime)\mathrm{d}\bar{z}^\prime_k \wedge \mathrm{d}z^\prime.
\end{align*}
We shall estimate these last four integrals separately, and we refer to them as \textbf{(I)},\textbf{(II)},\textbf{(III)}, and \textbf{(IV)}.  Assuming $\sigma < c^\sharp/(2|v^\prime|^2)$, and in view of $c^\sharp/2 - 2|v^\prime||z - Z(x^\prime,t)| > c^\sharp/4$, for every $x^\prime \in V_2$, we have that, for every $x^\prime \in V_2$ and $0 \leq \gamma \leq \sigma$, the following estimate holds true:
\begin{equation} \label{eq:estimate_exponential_deformation_FBI_decay_caracterization}
	\left| e^{i\zeta\cdot(Z(x,t) - Z(x^\prime,t) - i \gamma v^\prime)-\langle \zeta \rangle \langle Z(x,t) - Z(x^\prime,t) - i \gamma v^\prime \rangle^2}\right|  \leq e^{-c|\zeta||Z(x,t) - Z(x^\prime,t)|^2 - \frac{c^\sharp\gamma}{4}|\zeta|}
\end{equation}
In the integrals \textbf{(I)} and \textbf{(III)} we have that $|Z(x,t) - Z(x^\prime,t)|^2 \geq \rho$, for some $\rho > 0$, therefore we can estimate the absolute value of both integrals by a constant times $e^{-c \rho |\zeta|}$. In the integral \textbf{(II)} we have that $\gamma = \sigma$, thus we can estimate the absolute value of \textbf{(II)} by a constant times $e^{-\sigma \frac{c^\sharp}{4}|\xi|}$. In the last integral we shall use the decay of $|\bar{\partial} F |$, obtaining that for every $k \in \mathbb{Z}_+$,
\begin{align*}
	|\textbf{(IV)}| & \leq \text{Const} \cdot C_3^{k+1}  k!^{s-1} \int_0^\sigma e^{-\frac{|\xi|c^\sharp}{4}\gamma} (\gamma + \lambda)^k \mathrm{d}\lambda\\
	& \leq \text{Const} \cdot C_3^{k+1} k!^{s-1} \sum_{j = 0}^k \binom{k}{j} \lambda^{k-j} \int_0^\infty e^{-\frac{|\xi| c^\sharp}{4} \lambda} \lambda^j \mathrm{d}\lambda\\
	& \leq \text{Const} \cdot \frac{C_3^{k+1} (\lambda + 4/c^\sharp)^k}{|\xi|^k} k!^s\\
	& \leq \text{Const} \cdot \frac{C_3^{k+1} (\delta^\prime+ 4/c^\sharp)^k}{|\xi|^k} k!^s
\end{align*}
Now since all the four integrals can be estimated uniformly on $0 < \lambda < \delta^\prime$, we have that there exist $C,\epsilon>0$ such that 
\begin{equation*}
	| \mathfrak{F}[\chi u](t; Z(x,t), \zeta)| \leq C e^{-\epsilon |\zeta|^{\frac{1}{s}}} \eqno{\forall (x, t) \in V_4 \times W_2, \zeta \in \mathbb{R}\mathrm{T}^\prime_{\Sigma_t}\big|_{Z(x,t)}(\mathcal{C}).}
\end{equation*}

\vspace{0,4cm}
\textit{2}. $\Rightarrow$ \textit{1}.:

\vspace{0,4cm}
To prove that $(0,0,\xi_0,0)$ does not belong to the Gevrey-$s$ wave front set of $u$ we shall use the F.B.I. transform with respect to the hypo-analytic structure of maximal rank defined in section \ref{sec:co-rank_zero} combined with Theorem \ref{thm:FBI-antonio-nicholas}. To be more precise, we shall prove that there exist $\Psi \in \mathcal{C}^\infty_c(V\times W)$, with $\Psi \equiv 1$ in some neighborhood of the origin, and constants $C,\epsilon >0$, such that, for every $(x,t)$ in some open neighborhood of the origin, and every $(\xi,\eta)$ in some conic neighborhood of $(\xi_0,0)$, the following estimate holds true,
\begin{equation}\label{eq:FBI_maximal_to_prove}
	\left| \widetilde{\mathfrak{F}}[\Psi u](\widetilde{\Xi}(x,t),\widetilde{\Theta}) \right| \leq C e^{-\epsilon |\widetilde{\Theta}|^\frac{1}{s}},
\end{equation}
where $\widetilde{\Xi}(x,t) = (\widetilde{Z}(x,t), t) = (Z(x,t) - i\mathrm{d}_t\phi(0)t,t) $, $\widetilde{\Theta} = (\zeta, \tau - {}^t\widetilde{Z}_t(x,t)\zeta)$, and 
\begin{align*}
	\widetilde{\mathfrak{F}}[\Psi u](\widetilde{\Xi}(x,t),\widetilde{\Theta}) = \Big \langle (\Psi u)(x^\prime,t^\prime), &  e^{i\widetilde{\Theta} \cdot (\widetilde{\Xi}(x,t) - \widetilde{\Xi}(x^\prime,t^\prime)) - \langle \widetilde{\Theta} \rangle \langle \widetilde{\Xi}(x,t) - \widetilde{\Xi}(x^\prime,t^\prime) \rangle^2} \cdot \\
	& \cdot \Delta(\widetilde{\Xi}(x,t) - \widetilde{\Xi}(x^\prime, t^\prime), \widetilde{\Theta}) \det \widetilde{Z}_{x,t}(x^\prime,t^\prime) \Big
	\rangle_{\mathcal{D}^\prime},
\end{align*}
and we note that $\det \widetilde{Z}_{x,t} = \det \widetilde{\Xi}_{x,t}$. 

Let $\chi \in \mathcal{C}_c^\infty(V)$ be such that $0 \leq \chi \leq 1$ and $\chi \equiv 1$ in $V_1\Subset V$, let $\widetilde{V} \Subset V_1$, $\widetilde{W} \Subset W$ be an open neighborhood of the origin, let $\xi_0 \subset \mathcal{C} \subset \mathrm{R}^m \setminus 0$ be an open convex cone, and let $C,\tilde{\epsilon}>0$ such that 
\begin{equation}\label{eq:FBI-decay}
		    |\mathfrak{F}[\chi u](t; Z(x,t), \zeta)| \leq C e^{-\tilde{\epsilon} |\zeta|^{\frac{1}{s}}}, \qquad \forall t \in \widetilde{W}, x \in \widetilde{V}, \zeta \in \mathbb{R} \mathrm{T}^\prime_{\Sigma_t}\big|_{Z(x,t)}(\mathcal{C}).
	\end{equation}
By Theorem \ref{thm:FBI-inversion-formula}, we have that
\begin{equation*}
    \chi(x)u(x,t)=\lim_{\epsilon\to 0^+}\frac{1}{(2\pi^3)^\frac{m}{2}}\iint_{\mathbb{R} \mathrm{T}^\prime_{\Sigma_t} }  e^{i\zeta \cdot (Z(x,t)- z^\prime)- \langle \zeta \rangle \langle Z(x,t) - z^\prime \rangle^2 - \epsilon \langle \zeta \rangle^2} \mathfrak{F}[\chi u] (t; z^\prime, \zeta) \langle \zeta \rangle^\frac{m}{2} \mathrm{d} z^\prime \mathrm{d}\zeta,
\end{equation*}
for $(x,t) \in V_1 \times W$, where we are denoting $\Sigma_t = Z(\mathbb{R}^m, t)$. For every $\epsilon>0$ we set
\begin{equation*}
    F^\epsilon(x,t) \doteq \frac{1}{(2\pi^3)^\frac{m}{2}}\iint_{\mathbb{R}\mathrm{T}^\prime_{\Sigma_t}(\mathcal{C})}e^{i\zeta\cdot(Z(x,t)-z^\prime)-\langle\zeta\rangle\langle Z(x,t)-z^\prime\rangle^2-\epsilon\langle\zeta\rangle^2}\mathfrak{F}[\chi u](t;z^\prime,\zeta)\langle\zeta\rangle^\frac{m}{2}\mathrm{d}z^\prime \mathrm{d}\zeta,
\end{equation*}
and
\begin{equation*}
    g^\epsilon(z,t) \doteq \frac{1}{(2\pi^3)^\frac{m}{2}}\iint_{\mathbb{R}\mathrm{T}^\prime_{\Sigma_t}(\mathbb{R}^m \setminus \mathcal{C})}e^{i\zeta\cdot(z-z^\prime)-\langle\zeta\rangle\langle z -z^\prime\rangle^2-\epsilon\langle\zeta\rangle^2}\mathfrak{F}[\chi u](t;z^\prime,\zeta)\langle\zeta\rangle^\frac{m}{2}\mathrm{d}z^\prime \mathrm{d}\zeta.
\end{equation*}
We shall deal with these two functions separately. We begin with $F^\epsilon$. Let $V^\prime$ be any neighborhood of the origin compactly contained in $\widetilde{V}$.  There exists $\rho>0$ such that, for $t \in \widetilde{W}$ and $z \in Z(V^\prime,t)$, 
\begin{equation*}
        \left| e^{i\zeta\cdot(z-z^\prime)-\langle\zeta\rangle\langle z -z^\prime\rangle^2-\epsilon\langle\zeta\rangle^2}\mathfrak{F}[\chi u](t;z^\prime,\zeta)\langle\zeta\rangle^\frac{m}{2} \right| \leq 
        \text{Const}\cdot e^{-\rho|\zeta| - \kappa^\prime|\zeta||z - z^\prime |^2/2},
        \quad(z^\prime,\zeta) \in \mathbb{R}\mathrm{T}^\prime_{\Sigma_t\setminus Z(\widetilde{V},t)}(\mathcal{C}),
\end{equation*}
and
\begin{equation*}
        \left| e^{i\zeta\cdot(z-z^\prime)-\langle\zeta\rangle\langle z -z^\prime\rangle^2-\epsilon\langle\zeta\rangle^2}\mathfrak{F}[\chi u](t;z^\prime,\zeta)\langle\zeta\rangle^\frac{m}{2} \right| \leq
        C |\zeta|^\frac{m}{2} e^{-\tilde{\epsilon} |\zeta|^{\frac{1}{s}}} e^{- \kappa^\prime|\zeta||z - z^\prime|^2},
       \quad (z^\prime,\zeta) \in \mathbb{R}\mathrm{T}^\prime_{Z(\widetilde{V},t)}(\mathcal{C}),
\end{equation*}
for all $k \in \mathbb{Z}_+$.  Combining these two estimates we obtain
\begin{equation*}
        \left| e^{i \zeta \cdot (z - z^\prime) -\langle \zeta \rangle \langle z - z^\prime \rangle^2 - \epsilon \langle \zeta \rangle^2} 
        \mathfrak{F}[\chi u](t;z^\prime,\zeta) \langle \zeta \rangle^{m/2} \right| \leq
        \text{Const} |\zeta|^\frac{m}{2} e^{-\tilde{\epsilon} |\zeta|^{\frac{1}{s}}} e^{-\kappa^\prime|\zeta||z-z^\prime|^2/2},
\end{equation*}
for all $t \in \widetilde{W}$, $z \in Z(V^\prime,t)$, $(z^\prime,\zeta) \in \mathbb{R}\mathrm{T}^\prime_{\Sigma_t}(\mathcal{C})$, and $k \in \mathbb{Z}_+$. In particular, this means that the function 
\begin{equation*}
	F(x,t) = \frac{1}{(2\pi^3)^\frac{m}{2}}\iint_{\mathbb{R}\mathrm{T}^\prime_{\Sigma_t}(\mathcal{C})}e^{i\zeta\cdot(Z(x,t)-z^\prime)-\langle\zeta\rangle\langle Z(x,t)-z^\prime\rangle^2}\mathfrak{F}[\chi u](t;z^\prime,\zeta)\langle\zeta\rangle^\frac{m}{2}\mathrm{d}z^\prime \mathrm{d}\zeta
\end{equation*}
is continuos in $V^\prime \times \widetilde{W}$ by Lebesgue's dominated convergence Theorem. 
So let $(x, t) \in V^\prime \times \widetilde{W}$. If for every multi-indices $\alpha, \beta, \ell_1, \ell_2 \in \mathbb{Z}_+^m$ with $\beta \leq \alpha$ and $\ell_1 + 2\ell_2 = \beta$ we choose $k = |\alpha-\beta + \ell_1 + \ell_2| + \kappa$, where $\kappa$ is any integer bigger than $3m/2+1$, then, by Lemma 3.6 of \cite{braun2022} (or Fa\`a di Bruno's formula),
\begin{align*}
   \left| \mathrm{M}^\alpha F^\epsilon(x,t)\right|&\leq \text{Const} \sum_{\beta\leq\alpha}\binom{\alpha}{\beta}\sum_{\ell=(\ell^\prime,\ell^{\prime\prime})}\frac{\beta!}{\ell!}\iint_{\mathbb{R}\mathrm{T}^\prime_{\Sigma_t}(\mathcal{C})}e^{-\tilde{\epsilon}|\zeta|^{\frac{1}{s}}}e^{-\kappa^\prime|\zeta||Z(x,t)-z^\prime|^2/2} \cdot\\
   &\cdot|\zeta|^{|\alpha-\beta|+\ell^1_1+\ell^1_2+\cdots+ \ell^m_1+\ell^m_2+\frac{m}{2}}|\mathrm{d}z^\prime \mathrm{d}\zeta|\\
   &\leq  \text{Const}^{|\alpha|+1}\sum_{\beta\leq\alpha}\binom{\alpha}{\beta}\sum_{\ell=(\ell^\prime,\ell^{\prime\prime})}\frac{\beta!}{\ell!}\frac{\alpha!^s}{\beta!^s}(\ell^1_1+\ell^1_2)!^s\cdots(\ell^m_1+\ell^m_2)!^s\\
   &\leq  \text{Const}^{|\alpha|+1}\alpha!^s,
\end{align*}
where the constant $\text{Const}$ does not depend on $\epsilon$.  In conclusion, we have that the functions $\mathrm{M}^\alpha F(x,t)$ are continuous in $V^\prime \times \widetilde{W}$ for all $\alpha \in \mathbb{Z}_+^m$, and
\begin{equation}\label{eq:v_1_Gevrey_partial_FBI_proof_2-1}
    \sup_{(x,t) \in V^\prime \times \widetilde{W}}\left| \mathrm{M}^\alpha F(x,t) \right| \leq C_1^{|\alpha| + 1} \alpha!^s, \qquad \forall \alpha \in \mathbb{Z}_+^m,
\end{equation}
for some $C_1>0$. This implies (see the proof of Theorem 1.1 of \cite{braundasilva:2022}) that there exists $G \in \mathcal{C}^0(\widetilde{W};\mathcal{C}^\infty(V^\prime + i\mathbb{R}^m))$ and $C_{G} > 0$ such that
\begin{equation}
    \begin{cases}
        G(Z(x,t), t) = F(x,t);  \\
        |\bar{\partial}_z G(z,t)| \leq C_{G}^{k+1} k!^{s-1} |\Im z - \phi(\Re z, t)|^k,\quad \forall k \in \mathbb{Z}_+,  |\Im z - \phi(\Re z, t)| \ll 1.
    \end{cases}
\end{equation}
Now we shall deal with $g^\epsilon$: We begin with the following standard decomposition of $\mathbb{R}^m \setminus \mathcal{C}$,
\begin{equation*}
    \mathbb{R}^m \setminus \mathcal{C} = \cup_{j=1}^N \overline{\mathcal{C}}_j, 
\end{equation*}
where each $\mathcal{C}_j \subset \mathbb{R}^m \setminus 0$ is an acute open convex cone, satisfying
\begin{enumerate}
    \item $\mathcal{C}_j \cap \mathcal{C}_k = \empty$ for all $j\neq k$;\\
    \item for every $j = 1, \dots, N$, the set 
        \begin{equation*}
            \Gamma_j \doteq \{ v \in \mathbb{R}^m \; : \; v \cdot \xi >0, \forall \xi \in \overline{\mathcal{C}_j}, \; \text{and}\; v \cdot \xi_0 < 0 \}
        \end{equation*}
    is a nonempty acute open convex cone.
\end{enumerate}
Shrinking $\Gamma_j$ if necessary, we can assume that, for all $j$, and $(v,\xi) \in \Gamma_j \times  \mathcal{C}_j$,
\begin{equation*}
    v \cdot \xi \geq 2 c^\sharp |\xi| |v|,
\end{equation*}
for some $c^\sharp > 0$. Furthermore, there exist $V_2 \subset V^\prime$ and $W_2 \subset \widetilde{W}$ open neighborhoods of the origin such that
\begin{equation*}
   \hspace{5.5cm} \Re \zeta \cdot v \geq c^\sharp |\zeta| |v|, \eqno{t \in W_2, (v,\zeta) \in \Gamma_j \times \mathbb{R}\mathrm{T}^\prime_{\Sigma_t}|_{V_2}(\mathcal{C}_j),}
\end{equation*}
%
for every $j$. So we make the following decomposition of $g^\epsilon(z,t)$:
\begin{equation*}
g^\epsilon(z,t) = \sum_{j=1}^N g_{j}^\epsilon(z,t),
\end{equation*}
where 
\begin{equation*}
    g_{j}^\epsilon(z,t) \doteq \frac{1}{(2\pi^3)^\frac{m}{2}}\iint_{\mathbb{R}\mathrm{T}^\prime_{\Sigma_t}(\mathcal{C}_j)}e^{i\zeta\cdot(z-z^\prime)-\langle\zeta\rangle\langle z-z^\prime\rangle^2-\epsilon\langle\zeta\rangle^2}\mathfrak{F}[\chi u](t;z^\prime,\zeta)\langle\zeta\rangle^\frac{m}{2}\mathrm{d}z^\prime\wedge\mathrm{d}\zeta.
\end{equation*}
Let $V^{\prime\prime} \Subset V_2$ be an open neighborhood of the origin. We claim that, for every $t \in W_2$, $g_{j}^\epsilon(z,t)$ converges uniformly to $g_{j}(z,t)$ in $\mathcal{W}_{\delta}(Z(V^{\prime\prime},t),\Gamma_j) = \{Z(x,t) + iv \; : \;  x \in V^{\prime\prime}, v \in (\Gamma_j)_{\delta} \}$, for some $\delta > 0$, and consequently $g_j(z,t)$ is holomorphic with respect to $z$ in $\mathcal{W}_{\delta^\prime}(Z(V^{\prime\prime},t),\Gamma_j)$ for every $t\in W_2$. Since $V^{\prime\prime} \Subset V_2$, there exists $\eta > 0$ such that 
\begin{equation*}
	|Z(x,t) - Z(x^\prime,t)| > \eta,\eqno{\forall t \in W_2, x \in V^{\prime\prime}, x^\prime \notin V_2.}
\end{equation*}
Let $0 < \delta < \kappa^\prime c^\sharp /(2(\kappa^\prime + 1)) $ be such that
\begin{equation}\label{eq:condition_delta_FBI_proof}
        \frac{\delta + \sqrt{\delta^2 + \kappa^\prime(\delta^2 + \delta)}}{\kappa^\prime} < \eta.
\end{equation}
Let $j \in \{1, \dots, \ell\}$, and $v \in (\Gamma_j)_\delta$ be fixed, and let $t \in W_2$ and $z \in Z(V^{\prime\prime},t)$. For every $(z^\prime,\zeta) \in \mathbb{R}\mathrm{T}^\prime_{\Sigma_t}(\mathcal{C}_j)$ we have that
\begin{equation*}
        \left| e^{i \zeta \cdot (z + iv - z^\prime) - \langle \zeta \rangle \langle z + iv - z^\prime \rangle^2} \right| 
        \leq
        e^{-\kappa^\prime|\zeta||z - z^\prime|^2 - \Re \, v \cdot \zeta + |\zeta|(2|v||z - z^\prime| + |v|^2)}.
\end{equation*}
%
\comment{
\begin{align*}
    \Re \{ i\zeta \cdot (z - z^\prime) - \langle \zeta \rangle \langle z - z^\prime \rangle^2 \} & = \Re \{ i \zeta \cdot (Z(x,t) + iv - z^\prime) - \langle \zeta \rangle \langle Z(x,t) + iv - z^\prime \rangle^2 \}\\
    & = - \Re \zeta \cdot v + \Re \{ i \zeta \cdot (Z(x,t) - z^\prime) - \langle \zeta \rangle \langle Z(x,t) - z^\prime \rangle^2 \}\\
    & - \Re \{ \langle \zeta \rangle \big( 2i v \cdot (Z(x,t) - z^\prime) - |v|^2 \big) \}\\
    & \leq - \frac{c^\sharp}{2}|v||\zeta| - c|\zeta||Z(x,t) - z^\prime|^2 + |\zeta|\big( |v||Z(x,t) - z^\prime| + |v|^2\big)\\
    & = -|\zeta| \big( |v|(c^\sharp/2 - |v|) + c|Z(x,t) - z^\prime|(|Z(x,t) - z^\prime| - |v|) \big).
\end{align*}
}
%
Now since $|Z(x,t) - z^\prime|(\kappa^\prime|Z(x,t) - z^\prime| - 2|v|) > - |v|^2/\kappa^\prime$, and in view of the choose of $\delta$, if $(z^\prime,\zeta) \in \mathbb{R}\mathrm{T}^\prime_{Z(V_2,t)}(\mathcal{C}_j)$, then
%
%
\begin{equation*}
        \left| e^{i \zeta \cdot (z + iv - z^\prime) - \langle \zeta \rangle \langle z + iv - z^\prime \rangle^2} \right| 
        \leq
        e^{-c^\sharp|v||\zeta|/2}.
\end{equation*}
%
\comment{
\begin{equation*}
    \Re \{ i\zeta \cdot (z - z^\prime) - \langle \zeta \rangle \langle z - z^\prime \rangle^2 \} < - \frac{c^\sharp|v|}{8}|\zeta|.
\end{equation*}
}
%
In vier of \eqref{eq:condition_delta_FBI_proof}, there exists $\varepsilon^\prime > 0$ such that, if $(z^\prime,\zeta) \in \mathbb{R}\mathrm{T}^\prime_{Z(\mathbb{R}^m \setminus V_2,t)}(\mathcal{C}_j)$, then
%
\comment{
\begin{align*}
    \Re \{ i\zeta \cdot (z - z^\prime) - \langle \zeta \rangle \langle z - z^\prime \rangle^2 \}
    & = - \Re \zeta \cdot v + \Re \{ i \zeta \cdot (Z(x,t) - z^\prime) - \langle \zeta \rangle \langle Z(x,t) - z^\prime \rangle^2 \}\\
    & - \Re \{ \langle \zeta \rangle \big( 2i v \cdot (Z(x,t) - z^\prime) - |v|^2 \big) \}\\
    & \leq |v||\zeta| - c|\zeta||Z(x,t) - z^\prime|^2 + |\zeta|\big( |v||Z(x,t) - z^\prime| + |v|^2\big)\\
    & = -|\zeta| \big( c|Z(x,t) - z^\prime|^2 -|v||Z(x,t) - z^\prime| - |v|(1+|v|)) \big).
\end{align*}
}
%
\begin{align*}
	 \left| e^{i \zeta \cdot (z + iv - z^\prime) - \langle \zeta \rangle \langle z + iv - z^\prime \rangle^2} \right| & \leq e^{-\kappa^\prime|z-z^\prime|^2|\zeta| + 2\delta|z-z^\prime||\zeta| + (\delta^2 + \delta)|\zeta|} \\ 
	 & \leq e^{-\varepsilon^\prime|\zeta|}.
\end{align*}
%
\comment{
Shrinking $\delta^\prime$ we can assume that $\frac{|v| + \sqrt{|v|^2 - 4c|v|(1+|v|)}}{2c} < \frac{r^\prime}{4}$, and in view of $|x| < r^\prime / 2$, we have that 
\begin{equation*}
    \Re \{ i\zeta \cdot (z - z^\prime) - \langle \zeta \rangle \langle z - z^\prime \rangle^2 \} < - \tilde{c}|\zeta|,
\end{equation*}
for some $\tilde{c}>0$.
}
%
So, for every $t \in W_2$, $g_{j}^\epsilon(z,t)$ is uniformly bounded in $\mathcal{W}_{\delta}(Z(V^{\prime\prime},t),\Gamma_j)$, thus, by Montel's Theorem, there exists a sequence $\epsilon_\nu$ going to zero such that $g_{j}^{\epsilon_\nu}(z,t)$ converges uniformly to $g_{j}(z,t)$ in $\mathcal{W}_{\delta}(Z(V^{\prime\prime},t),\Gamma_j)$, consequently $g_j(z,t)$ is holomorphic with respect to $z$ in $\mathcal{W}_{\delta^\prime}(Z(V^{\prime\prime},t),\Gamma_j)$ for every $t\in W_2$. 
%
\comment{
o a holomorphic function on $\mathcal{W}_{t\delta^\prime}(B_{r^\prime / 2}(0),\Gamma_j)$, given by
\begin{equation*}
    g_{j}(z,t) \doteq \frac{1}{(2\pi^3)^\frac{m}{2}}\iint_{\mathbb{R}\mathrm{T}^\prime_{\Sigma_t}(\mathcal{C}_j)}e^{i\zeta\cdot(z-z^\prime)-\langle\zeta\rangle\langle z-z^\prime\rangle^2}\mathfrak{F}[\chi u](t;z^\prime,\zeta)\langle\zeta\rangle^\frac{m}{2}\mathrm{d}z^\prime \mathrm{d}\zeta.
\end{equation*}
}
%
Following the argument in the end of the proof of Theorem 1.4 of \cite{braundasilva:2022}, we obtain the following decomposition of $(\chi u)(\cdot,t)$ on $\mathcal{D}^\prime(V^{\prime\prime})$ for every $t \in W_2$:
\begin{equation*}
	\chi u(\cdot,t) = G(Z(\cdot,t), t) + \sum_{j=1}^N \mathrm{bv}\big( g_{j}(\cdot,t)\big).
\end{equation*}
Now let $\psi \in \mathcal{C}^\infty_c(W_2)$ satisfying $0 \leq \psi \leq 1$, and $\psi \equiv 1$ in some open neighborhood of the origin, and let $\widetilde{\chi} \in \mathcal{C}^\infty_c(V^{\prime\prime})$, satisfying $0 \leq \widetilde{\chi} \leq 1$, and $\widetilde{\chi} \equiv 1$ in $V_2 \Subset V^{\prime\prime}$. Then, writing $\Xi(x,t) = (\widetilde{Z}(x,t),t)$, and $\widetilde{\Theta} = (\zeta, \tau - {}^{\mathrm{t}}\widetilde{Z}_t(x,t)\zeta)$, as in section \ref{sec:co-rank_zero},
\begin{align}\label{eq:decomposition_enlarged_FBI_2_implies_1}
	&\widetilde{\mathfrak{F}}[(\widetilde{\chi} \otimes \psi) u](\Xi(x,t), \widetilde{\Theta}) = \\ \nonumber
	&= \iint_{V^{\prime\prime} \times W_2} e^{i \widetilde{\Theta} \cdot (\Xi(x,t) - \Xi(x^\prime,t^\prime))  - \langle \widetilde{\Theta} \rangle \langle \Xi(x,t) - \Xi(x^\prime,t^\prime \rangle^2}\widetilde{\chi}(x^\prime)\psi(t^\prime) G(Z(x^\prime,t^\prime), t^\prime) \cdot \\ \nonumber
	& \qquad \cdot \Delta(\widetilde{\Xi}(x,t) - \widetilde{\Xi}(x^\prime, t^\prime), \widetilde{\Theta}) \det Z_x(x^\prime, t^\prime) \mathrm{d}x^\prime \mathrm{d}t^\prime\\ \nonumber
	&+\lim_{\lambda \to 0} \iint_{V^{\prime\prime} \times W_2} e^{i \widetilde{\Theta} \cdot (\Xi(x,t) - \Xi(x^\prime,t^\prime))  - \langle \widetilde{\Theta} \rangle \langle \Xi(x,t) - \Xi(x^\prime,t^\prime \rangle^2}\widetilde{\chi}(x^\prime)\psi(t^\prime) g_{1}(Z(x^\prime,t^\prime) + i\lambda v_1,t^\prime) \cdot \\ \nonumber
	& \qquad \cdot \Delta(\widetilde{\Xi}(x,t) - \widetilde{\Xi}(x^\prime, t^\prime), \widetilde{\Theta}) \det Z_x(x^\prime, t^\prime)  \mathrm{d}x^\prime \mathrm{d}t^\prime\\ \nonumber
	&\quad \vdots\\ \nonumber
	&+\lim_{\lambda \to 0} \iint_{V^{\prime\prime} \times W_2} e^{i \widetilde{\Theta} \cdot (\Xi(x,t) - \Xi(x^\prime,t^\prime))  - \langle \widetilde{\Theta} \rangle \langle \Xi(x,t) - \Xi(x^\prime,t^\prime \rangle^2}\widetilde{\chi}(x^\prime)\psi(t^\prime) g_{N}(Z(x^\prime,t^\prime) + i\lambda v_N,t^\prime) \cdot \\ \nonumber
	& \qquad \cdot \Delta(\widetilde{\Xi}(x,t) - \widetilde{\Xi}(x^\prime, t^\prime), \widetilde{\Theta}) \det Z_x(x^\prime, t^\prime) \mathrm{d}x^\prime \mathrm{d}t^\prime,
\end{align}
where $v_j \in \Gamma_j$, for $j = 1, \dots, N$.  
Let $V^{\prime\prime\prime} \Subset V_2$ and let $0<c_0<1$. We shall prove that there exists $\tilde{\mathcal{C}} \subset \mathbb{R}^m\setminus \{0\}$ an open conic neighborhood of $\xi_0$, and constants $C,\epsilon > 0$, such that, for $(\xi,\tau) \in \{(\xi^\prime, \tau^\prime) \in \mathbb{R}^{m+n} \setminus \{0\} \; : \; |\tau| < c_0 |\xi|, \xi \in \tilde{\mathcal{C}}$, and $(x,t) \in V^{\prime\prime\prime} \times W_2$, the estimate \eqref{eq:FBI_maximal_to_prove} holds true. We shall only estimate the contribution of the first tem on the right-hand side of \eqref{eq:decomposition_enlarged_FBI_2_implies_1}, since the others can be estimated in a similar fashion following the \textit{1.} $\Rightarrow$ \textit{2.} part of the present proof. We stress the fact that the existence of $\tilde{\mathcal{C}}$ comes from this part that we shall omit here (but it follows from the same argument form the \textit{1.} $\Rightarrow$ \textit{2.}).  Set $\widetilde{G}(z,t) = G(z + i \mathrm{d}_t \phi(0)t, t)$. Then for every $(x,t) \in V^{\prime\prime} \times W_2$, we have that
\begin{align*}
    \widetilde{G}(\widetilde{Z}(x,t),t) & = G(\widetilde{Z}(x,t) + i \mathrm{d}_t \phi(0)t,t) \\
    & = G(Z(x,t) - i\mathrm{d}_t \phi(0)t + i \mathrm{d}_t \phi(0)t,t) \\
    & = G(Z(x,t), t),
\end{align*}
and we also have that, for every $k \in \mathbb{Z}_+$ and $|\Im z + \mathrm{d}_t \phi(0)t - \phi(\Re z, t)| \ll 1$,
\begin{align*}
    |\bar{\partial}_z \widetilde{G}(z, t)| & = |\bar{\partial}_z G(z + i \mathrm{d}_t \phi(0)t,t)| \\
    & \leq C_G^{k+1} k!^{s-1} |\Im z + \mathrm{d}_t \phi(0)t - \phi(\Re z, t)|^k.
\end{align*}
We then write
\begin{align}
\nonumber	& \iint_{V^{\prime\prime} \times W_2} e^{i \widetilde{\Theta} \cdot (\Xi(x,t) - \Xi(x^\prime,t^\prime))  - \langle \widetilde{\Theta} \rangle \langle \Xi(x,t) - \Xi(x^\prime,t^\prime \rangle^2}\widetilde{\chi}(x^\prime)\psi(t^\prime) G(Z(x^\prime,t^\prime), t^\prime) \cdot \\ \nonumber
	& \qquad \cdot \Delta(\widetilde{\Xi}(x,t) - \widetilde{\Xi}(x^\prime, t^\prime), \widetilde{\Theta}) \det Z_x(x^\prime, t^\prime) \mathrm{d}x^\prime \mathrm{d}t^\prime = \\
	&= \iint_{V_2 \times W_2} e^{i \widetilde{\Theta} \cdot (\Xi(x,t) - \Xi(x^\prime,t^\prime))  - \langle \widetilde{\Theta} \rangle \langle \Xi(x,t) - \Xi(x^\prime,t^\prime \rangle^2}\psi(t^\prime) \widetilde{G}(\widetilde{Z}(x^\prime,t^\prime), t^\prime) \cdot \label{eq:first_term_end_proof_2-1} \\ \nonumber
	& \qquad \cdot \Delta(\widetilde{\Xi}(x,t) - \widetilde{\Xi}(x^\prime, t^\prime), \widetilde{\Theta}) \det Z_x(x^\prime, t^\prime) \mathrm{d}x^\prime \mathrm{d}t^\prime\\ \
	&+\iint_{(V^{\prime\prime}\setminus V_2) \times W_2} e^{i \widetilde{\Theta} \cdot (\Xi(x,t) - \Xi(x^\prime,t^\prime))  - \langle \widetilde{\Theta} \rangle \langle \Xi(x,t) - \Xi(x^\prime,t^\prime \rangle^2}\widetilde{\chi}(x^\prime)\psi(t^\prime) G(Z(x^\prime,t^\prime), t^\prime) \cdot  \label{eq:second_term_end_proof_2-1} \\ \nonumber
	& \qquad \cdot \Delta(\widetilde{\Xi}(x,t) - \widetilde{\Xi}(x^\prime, t^\prime), \widetilde{\Theta}) \det Z_x(x^\prime, t^\prime) \mathrm{d}x^\prime \mathrm{d}t^\prime.
\end{align}
To estimate the absolute value of \eqref{eq:second_term_end_proof_2-1} we note that there exists $r> 0$ such that, if $(x,t) \in V^{\prime\prime\prime} \times W_2$ and $(x^\prime,t^\prime) \in (V^{\prime\prime}\setminus V_2) \times W_2$, then
\begin{align*}
   \big| e^{i \widetilde{\Theta} \cdot (\Xi(x,t) - \Xi(x^\prime,t^\prime))  - \langle \widetilde{\Theta} \rangle \langle \Xi(x,t) - \Xi(x^\prime,t^\prime \rangle^2} \big| & \leq e^{-r\tilde{\kappa}^\prime|\widetilde{\Theta}|},
\end{align*}
thus we can estimate the absolute value of \eqref{eq:second_term_end_proof_2-1} by a constant times $e^{-r\tilde{\kappa}^\prime|\widetilde{\Theta}|}$. To estimate the absolute vale of \eqref{eq:first_term_end_proof_2-1} we first use Stokes' Theorem to write (keeping in mind that $Z_x = \widetilde{Z}_x$)
\begin{align}
   \nonumber & \int_{V_2}  e^{i \widetilde{\Theta} \cdot (\Xi(x,t) - \Xi(x^\prime,t^\prime))  - \langle \widetilde{\Theta} \rangle \langle \Xi(x,t) - \Xi(x^\prime,t^\prime \rangle^2}\widetilde{G}(Z(x^\prime,t^\prime), t^\prime) \cdot \\ \nonumber
   & \qquad \cdot \Delta(\widetilde{\Xi}(x,t) - \widetilde{\Xi}(x^\prime, t^\prime), \widetilde{\Theta}) \det \widetilde{Z}_x(x^\prime, t^\prime) \mathrm{d}x^\prime = \\ 
    & = \int_{\widetilde{Z}(V_2,t^\prime) - i\lambda_0 \zeta / \langle \zeta \rangle} e^{i \widetilde{\Theta} \cdot (\Xi(x,t) - (z,t^\prime))  - \langle \widetilde{\Theta} \rangle \langle \Xi(x,t) - (z,t^\prime) \rangle^2} \widetilde{G}(z, t^\prime) \cdot \label{eq:eq:stokes_enlarged_FBI_2_implies_1_integral_1} \\ \nonumber
    & \qquad \cdot \Delta(\widetilde{\Xi}(x,t) - (z, t^\prime), \widetilde{\Theta}) \mathrm{d}z \\ 
    & - \int_{\widetilde{Z}(\partial V_2,t^\prime) - i[0,\lambda_0] \zeta / \langle \zeta \rangle} e^{i \widetilde{\Theta} \cdot (\Xi(x,t) - (z,t^\prime))  - \langle \widetilde{\Theta} \rangle \langle \Xi(x,t) - (z,t^\prime) \rangle^2} \widetilde{G}(z, t^\prime) \cdot \label{eq:eq:stokes_enlarged_FBI_2_implies_1_integral_2} \\ \nonumber
    & \qquad \cdot \Delta(\widetilde{\Xi}(x,t) - (z, t^\prime), \widetilde{\Theta}) \mathrm{d}z  \\
    & + \sum_{j=1}^m \int_{\widetilde{Z}(V_2,t^\prime) - i[0, \lambda_0] \zeta / \langle \zeta \rangle} e^{i \widetilde{\Theta} \cdot (\Xi(x,t) - (z,t^\prime))  - \langle \widetilde{\Theta} \rangle \langle \Xi(x,t) - (z,t^\prime) \rangle^2} \bar{\partial}_{z_j}\widetilde{G}(z, t^\prime) \cdot \label{eq:eq:stokes_enlarged_FBI_2_implies_1_integral_3} \\ \nonumber
    & \qquad \cdot \Delta(\widetilde{\Xi}(x,t) - (z, t^\prime), \widetilde{\Theta}) \mathrm{d}\bar{z}_j \wedge \mathrm{d}z ,
\end{align}
for some small $\lambda_0 > 0$. Reasoning as before, the absolute value of \eqref{eq:eq:stokes_enlarged_FBI_2_implies_1_integral_2} is bounded by a constant times $e^{-r\tilde{\kappa}^\prime|\widetilde{\Theta}|}$. Now writing $z = \widetilde{Z}(x^\prime, t^\prime) - i \sigma \zeta / \langle \zeta \rangle$ and $0 \leq \sigma \leq \lambda_0$,
\begin{align*}
    \Re \big\{ i \widetilde{\Theta} \cdot & (\Xi(x,t) - (z,t^\prime))  - \langle \widetilde{\Theta} \rangle \langle \Xi(x,t) - (z,t^\prime) \rangle^2 \big\} = \\
    & = \Re \big\{ i \widetilde{\Theta} \cdot (\Xi(x,t) - \Xi(x^\prime,t^\prime) + i(\sigma \zeta/\langle \zeta \rangle, 0))  - \langle \widetilde{\Theta} \rangle \langle \Xi(x,t) - \Xi(x^\prime,t^\prime) + i(\sigma \zeta/\langle \zeta \rangle, 0) \rangle^2 \big\} \\
    & = \Re \big\{ i \widetilde{\Theta} \cdot (\Xi(x,t) - \Xi(x^\prime,t^\prime))  - \langle \widetilde{\Theta} \rangle \langle \Xi(x,t) - \Xi(x^\prime,t^\prime) \rangle^2 \big\} - \sigma \Re \widetilde{\Theta} \cdot (\zeta/\langle \zeta \rangle, 0) \\
    &  - \Re\big\{  \langle \widetilde{\Theta} \rangle \big( 2i\sigma (\zeta/\langle \zeta \rangle, 0)\cdot (\Xi(x,t) - \Xi(x^\prime,t^\prime)) - \sigma^2\langle (\zeta/\langle \zeta \rangle, 0) \rangle^2\big) \big\} \\
    & \leq - \tilde{\kappa}^\prime|\widetilde{\Theta}| |\Xi(x,t) - \Xi(x^\prime,t^\prime)|^2 - \sigma \Re \langle \zeta \rangle + 2\sigma|\widetilde{\Theta}| \left|\frac{\zeta}{\langle \zeta \rangle}\right||\widetilde{Z}(x,t) - \widetilde{Z}(x^\prime,t^\prime)| + \sigma^2|\widetilde{\Theta}|,
\end{align*}
for $\Xi(x,t) = (\widetilde{Z}(x,t), t)$ and $\widetilde{\Theta} = (\zeta, \tau - {}^{\mathrm{t}}\widetilde{Z}_t(x,t)\zeta)$. In view of \eqref{eq:comparison<zeta>|zeta|} and \eqref{eq:comparison-Theta}, and assuming assume $\lambda_0 < \min\{ \big(\frac{1-\kappa^2}{1+\kappa^2}\big)^{3/2}\frac{\tilde{\kappa}^\prime}{\sqrt{1+(c_1+c_0C_1)^2}} , \sqrt{\frac{1-\kappa^2}{1+\kappa^2}}\frac{1}{2\sqrt{1+(c_1+c_0C_1)^2}}$, 
we have that, for $0 \leq \sigma \leq \lambda_0$, 
\begin{align*}
    \Re \big\{ i \widetilde{\Theta} \cdot & (\Xi(x,t) - (z,t^\prime))  - \langle \widetilde{\Theta} \rangle \langle \Xi(x,t) - (z,t^\prime) \rangle^2 \big\}  \leq \\
    & \leq  -|\widetilde{\Theta}| \bigg(|\widetilde{Z}(x,t) - \widetilde{Z}(x^\prime,t^\prime)|\bigg( \tilde{\kappa}^\prime|\widetilde{Z}(x,t) - \widetilde{Z}(x^\prime,t^\prime)| - 2 \sqrt{\frac{1+\kappa^2}{1-\kappa^2}} \sigma \bigg) + \tilde{\kappa}^\prime|t - t^\prime|^2\bigg) \\
    & - \sigma |\zeta| \bigg( \sqrt{\frac{1-\kappa^2}{1+\kappa^2}} - \sigma \sqrt{1+(c_1+c_0C_1)^2} \bigg) \\
    & \leq -|\widetilde{\Theta}| |\widetilde{Z}(x,t) - \widetilde{Z}(x^\prime,t^\prime)|\bigg( \tilde{\kappa}^\prime|\widetilde{Z}(x,t) - \widetilde{Z}(x^\prime,t^\prime)| - 2 \sqrt{\frac{1+\kappa^2}{1-\kappa^2}} \sigma \bigg)  \\
    & - \frac{\sigma |\zeta|}{2}\sqrt{\frac{1-\kappa^2}{1+\kappa^2}} \\
    & \leq |\widetilde{\Theta}|\frac{\sigma^2}{4\tilde{\kappa}^\prime}\frac{1+\kappa^2}{1-\kappa^2}  - \frac{\sigma |\zeta|}{2}\sqrt{\frac{1-\kappa^2}{1+\kappa^2}} \\
    & \leq  -\sigma |\zeta| \bigg(\frac{1}{2}\sqrt{\frac{1-\kappa^2}{1+\kappa^2}} - \sqrt{1 + (c_1 +c_0C_1)^2}\frac{\sigma}{4\tilde{\kappa}^\prime}\frac{1+\kappa^2}{1-\kappa^2}\bigg) \\
    & \leq - \frac{\sigma}{4}\sqrt{\frac{1-\kappa^2}{1+\kappa^2}} |\zeta|.
\end{align*}
Therefore the absolute value of \eqref{eq:eq:stokes_enlarged_FBI_2_implies_1_integral_1} is bounded by a constant times $e^{- \frac{\lambda_0}{4}\sqrt{\frac{1-\kappa^2}{1+\kappa^2}} |\zeta|}$. Following the arguments in the \textit{1}. $\Rightarrow$ \textit{2}. part of the proof, one can estimate the absolute value of \eqref{eq:eq:stokes_enlarged_FBI_2_implies_1_integral_3} by a constant times  $e^{-\varepsilon |\zeta|^{\frac{1}{s}}}$, for some $\varepsilon > 0$. In conclusion, there exists $\varepsilon^\prime > 0$ such that the absolute value of \eqref{eq:first_term_end_proof_2-1} by a constant times $e^{-\varepsilon^\prime |\widetilde{\Theta}|}$, in view of $|\widetilde{\Theta}|$ and $|\zeta$ being comparable. Summing up, we have proved that
\begin{align*}
    \Bigg| & \iint_{V^{\prime\prime} \times W_2} e^{i \widetilde{\Theta} \cdot (\Xi(x,t) - \Xi(x^\prime,t^\prime))  - \langle \widetilde{\Theta} \rangle \langle \Xi(x,t) - \Xi(x^\prime,t^\prime \rangle^2}\widetilde{\chi}(x^\prime)\psi(t^\prime) G(Z(x^\prime,t^\prime), t^\prime) \cdot \\ \nonumber
	& \qquad \cdot \Delta(\widetilde{\Xi}(x,t) - \widetilde{\Xi}(x^\prime, t^\prime), \widetilde{\Theta}) \det Z_x(x^\prime, t^\prime) \mathrm{d}x^\prime \mathrm{d}t^\prime  \Bigg| \leq \text{Const}. e^{-\varepsilon^\prime |\widetilde{\Theta}|^\frac{1}{s}}.
\end{align*}
\end{proof}
%
%
\subsection{A version for tube structures}
%
%
Suppose now that $\mathcal{V}$ is of tube type, \textit{i.e.} we can choose the function $\phi$ to be independent of $x$, so the real-structural bundle $\mathbb{R}\mathrm{T}^\prime_{\Sigma_t}$ can be written as $\{(Z(x,t), \xi); x \in \mathbb{R}^m, t \in W, \xi \in \mathbb{R}^m \}$. For tube structures we shall prove a "finer"-version of Theorem \eqref{thm:microlocal_Gevrey_FBI}.
\begin{defn}
Let $u \in \mathcal{C}^1 ( W ; \mathcal{E}^\prime(V) )$ and $0<\kappa\leq 1$. We define the partial $\kappa$-F.B.I. of $u$ as
\begin{equation}
	\mathfrak{F}_\kappa [u](t; z,\xi) = \left \langle u(y,t) , e^{ i \xi \cdot ( z - Z(y,t) ) - |\xi|^\kappa \langle z - Z(y,t) \rangle^2 } \Delta_\kappa ( z - Z(x,t) , \xi ) \right \rangle,
\end{equation}
for $t \in W$, $z \in \mathbb{C}^m$, and $\xi \in \mathbb{R}^m$, where the function $\Delta_\kappa(z,\xi)$ is the Jacobian of the map
\begin{equation*}
	\xi \mapsto \xi + i z |\xi|^\kappa.
\end{equation*}
\end{defn}
The partial $\kappa$-F.B.I has the same properties as the usual partial F.B.I. transform, including the inversion formula:
\begin{prop}\label{prop:FBI-kappa-inversion-formula}
	Let $u \in \mathcal{C}^1 ( W ; \mathcal{E}^\prime(V) )$ and $0<\kappa\leq 1$. Then
	\begin{equation*}
		u(x,t) = \lim_{\epsilon \to 0^+} \frac{1}{ ( 2 \pi^3 )^{\frac{m}{2}}} \iint_{\mathbb{R}^{2m}} e^{ i \xi \cdot ( Z(x,t) - Z(x^\prime,t) ) - | \xi |^\kappa \langle Z(x,t) - Z(x^\prime,t) \rangle^2 - \epsilon |\xi|^2}\mathfrak{F}_\kappa [u](t ; Z(x^\prime,t), \xi) |\xi|^{\frac{\kappa m}{2}} \mathrm{d} x^\prime \mathrm{d} \xi,
	\end{equation*}
	where the convergence takes place in $\mathcal{C}^1 ( W ; \mathcal{D}^\prime (V) )$.
\end{prop}
We note that in the tube case, the exponential in the F.B.I. transform can be expressed as
\begin{equation*}
	e^{ i \xi \cdot ( Z(x,t) - Z(x^\prime,t) ) - | \xi |^\kappa \langle Z(x,t) - Z(x^\prime,t) \rangle^2 } = e^{i \xi \cdot (x - x^\prime) - |\xi|^\kappa | x - x^\prime|^2}.
\end{equation*}
Since the "oscillatory" part of the exponential is now purely imaginary (\textit{i.e.} is in fact an oscillatory term), we can apply the same steps as in the proof of Theorem \ref{thm:microlocal_Gevrey_FBI}, to prove the following Theorems:
\begin{thm}\label{thm:microlocal_Gevrey_FBI_tube}
	Let $s > 1$ and $s^{-1} \leq \kappa \leq 1$. Let $u \in \mathcal{C}^1 (W; \mathcal{D}^\prime (V))$ and $\xi_0 \in \mathbb{R}^m \setminus 0$. Then the following are equivalent:

	\begin{enumerate}
		\item[\textit{1.}]  $(0,0,\xi_0,0) \notin \mathrm{WF}_s(u)$;\\
	 	\item[\textit{2.}] For every $\chi \in \mathcal{C}_c^\infty(V)$, with $0\leq \chi \leq 1$ and $\chi \equiv 1$ in some open neighborhood of the origin, there exist $\widetilde{V} \subset V$, $\widetilde{W} \subset W$, open balls centered at the origin, an open cone $\Gamma \subset \mathbb{R}^m \setminus 0$ containing $\xi_0$, and constants $C,\epsilon > 0$ such that
	\begin{equation}\label{eq:FBI-decay-gevrey-tube}
		|\mathfrak{F}_\kappa [\chi u](t; Z(x,t), \xi)| \leq C e^{-\epsilon |\zeta|^{\frac{1}{s}}}, \qquad \forall t \in \widetilde{W}, x \in \widetilde{V}, \xi \in \Gamma.
	\end{equation}
\end{enumerate}
\end{thm}
\begin{thm}\label{thm:microlocal_smooth_FBI_tube}
	Let $0 < \kappa \leq 1$, $u \in \mathcal{C}^1 (W; \mathcal{D}^\prime (V))$ and $\xi_0 \in \mathbb{R}^m \setminus 0$. Then the following are equivalent:
	\begin{enumerate}
		\item[\textit{1.}]  $(0,0,\xi_0,0) \notin \mathrm{WF} (u)$;\\
	 	\item[\textit{2.}] For every $\chi \in \mathcal{C}_c^\infty(V)$, with $0\leq \chi \leq 1$ and $\chi \equiv 1$ in some open neighborhood of the origin, there exist $\widetilde{V} \subset V$, $\widetilde{W} \subset W$, open balls centered at the origin, an open cone $\Gamma \subset \mathbb{R}^m \setminus 0$ containing $\xi_0$, such that for every $N > 0$ exists $C_N > 0$ satisfying
    	\begin{equation*}
		    |\mathfrak{F}_\kappa[\chi u](t; Z(x,t), \xi)| \leq \frac{C_N}{(1 + |\zeta|)^N}, \qquad \forall t \in \widetilde{W}, x \in \widetilde{V}, \xi \in \Gamma.
	\end{equation*}
\end{enumerate}
\end{thm}

%

\section{Microlocal hypoellipticity}\label{sec:hypoelliptic}

%
In this section we shall prove the main results of our work. Let $\mathcal{V}$ be a real-analytic locally integrable structure of tube type of rank $n$ on $\Omega \subset \mathbb{R}^{n+m}$, an open neighborhood of the origin. As in section \ref{sec:preliminaries}, we shall fix open neighborhoods of the origin $V \subset \mathbb{R}^n$ and $W \subset \mathbb{R}^m$, such that $\mathcal{V}$ is given on $V \times W$ by 
\begin{equation}\label{eq:L_j_tubo}
	\mathrm{L}_j = \frac{\partial}{\partial t_j} - i\sum_{k=1}^m \frac{\partial \phi_j}{\partial x_k} \frac{\partial}{\partial x_k}, \qquad j = 1, \dots, n,
\end{equation}
where the function $\phi : W \longrightarrow \mathbb{R}^m$ is real-analytic, and satisfies $\phi(0) = 0$. For every $u \in \mathcal{D}^\prime(V \times W)$, we define $\mathbb{L} u$ by
\begin{equation*}
	\mathbb{L}u = \sum_{j=1}^n \mathrm{L}_j u \mathrm{d} t_j.
\end{equation*}
Now we state the main definition for our result:
\begin{defn}\label{def:star}
Let $\xi_0 \in \mathbb{S}^{m-1}$. We say that $\phi$ satisfies condition $(\star)$ at $\xi_0$ if there exist  $W_0 \subset W$ an open neighborhood of the origin, $\Gamma \subset \mathbb{R}^m \setminus 0$ an open cone containing $\xi_0$, and constants $C_L>0$ and $\frac{1}{2} \leq \theta < 1$, such that %
\begin{itemize}
	\item For every $\xi \in \Gamma$ the map $ t \mapsto \Phi(t) \cdot \xi$ is open on $W_0$;
	\item There exist constants $C_L>0$ and $1/2 \leq \theta <1$ such that for every $t \in W_0$ and $\xi \in \Gamma \cap \mathbb{S}^{m-1}$,
		\begin{equation}\label{eq:lojasiewicz}
			|\phi(t) \cdot \xi |^\theta \leq C_L \| {}^{\mathrm{t}} \mathrm{d} \phi(t) \xi \|.\tag{LC}
		\end{equation}
\end{itemize}
\end{defn}
It is easy to see that the characteristic set of $\mathcal{V}$ on $V \times W$ is given by
\begin{equation*}
	\mathrm{Char}(\mathcal{V}) = \{ (x,t,\xi,\eta) \in \Omega \times (\mathbb{R}^{m+n} \setminus 0) \; : \; \eta = 0,\; {}^{\mathrm{t}} \mathrm{d} \Phi(t)\xi = 0\}. 
\end{equation*}
Now for each $\xi \in \mathbb{R}^m$ we set 
\begin{equation}\label{eq:defn_sigma_xi}
	\Sigma_\xi \doteq \{ t \in W \; : \; {}^{\mathrm{t}} \mathrm{d} \Phi(t) \cdot \xi = 0 \}.
\end{equation}
\begin{prop}\label{prop:curves}
Let $\xi_0 \in \mathbb{S}^{m-1}$. Suppose that $\phi$ satisfies $(\star)$ at $\xi_0$, and let $W_0 \subset W$, and $\Gamma \subset \mathbb{R}^m \setminus 0$ as in the definition above.  
Then there is a familie of continuous, piecewise real-analytic curves $\gamma_{\xi,t} : [0, \delta_0(\xi,t)] \longrightarrow \overline{W}$, where $\xi \in \gamma\cap\mathbb{S}^{m-1}$, $t\in W_0\setminus \Sigma_\xi$, such that $\gamma_{t,\xi}(\delta_0(\xi,t)) \in \partial W_0$, 
satisfying the following properties:
 \begin{enumerate}
	\item For every $W_1 \Subset W_0$ open neighborhood of the origin, there exists $\delta > 0$ such that $\delta_0(\xi,t)\geq \delta$, for every $\xi \in \Gamma \cap \mathbb{S}^{m-1}$, and $ t \in W_1\setminus \Sigma_{\xi}$;\\
	\item The length of the curves $\gamma_{\xi,t}$ is uniformly bounded for every $\xi \in \Gamma \cap \mathbb{S}^{m-1}$, and $ t \in W_0 \setminus \Sigma_{\xi}$; \\
     	\item If $\xi \in \Gamma \cap \mathbb{S}^{m-1}$, and $ 0 \leq \tau \leq \delta_0(\xi,t)$, then 
     	\begin{equation*}
     		\Phi(t) \cdot \xi - \Phi(\gamma_{\xi,t}(\tau)) \cdot \xi \geq \frac{1}{2} \big( (1-\theta) (8C_L)^{-1} \big)^{\frac{1}{1-\theta}} \tau^{\frac{1}{1-\theta}}
     	\end{equation*}
\end{enumerate} 
\end{prop}
\begin{proof}
Let $\xi_0 \in \mathbb{S}^{m-1}$ be such that the map $\phi$ satisfies condition $(\star)$ at $\xi_0$. Therefore there exist $V_0 \subset V$ and $W_0$ open balls centered at the origin, $\Gamma \subset \mathbb{R}^m \setminus 0$ an open cone containing $\xi_0$, and constants $C_L>0$ and $\frac{1}{2} \leq \theta < 1$ such that for every $\xi \in \Gamma \cap \mathbb{S}^{m-1}$ the map $ W_0 \ni t \mapsto \phi(t) \cdot \xi$ is open and for every $(t,\xi) \in W_0 \times (\Gamma \cap \mathbb{S}^{m-1})$,
\begin{equation}\label{eq:Loyasiewicz_xi}
\tag{LC}
 \big| \phi(t) \cdot \xi \big|^\theta \leq C_L \| {}^\mathrm{t}\mathrm{d} \phi(t) \xi \|.
 \end{equation}
 Now we set for every $\xi \in \mathbb{R}^m$,
\begin{equation*}
\Sigma_\xi = \big\{ t \in W \; : \;  {}^\mathrm{t} \mathrm{d} \Phi(t) \xi = 0 \big\}.
\end{equation*}
Fix $\xi \in \Gamma \cap \mathbb{S}^{m-1}$. For every $t \in  W_0 \setminus \Sigma_\xi$ we consider the following system of ordinary differential equations:
\begin{equation}
    \begin{cases}
    \frac{\mathrm{d} \alpha}{\mathrm{d} \tau}(\tau)=-\frac{ {}^\mathrm{t} \mathrm{d} \phi(\alpha(\tau)) \xi}{\|  {}^\mathrm{t} \mathrm{d} \phi(\alpha(\tau)) \xi \|};\\
    \alpha(0)=t,
    \end{cases}
\end{equation}
When necessary we shall write $\alpha_{\xi,t}$ instead of $\alpha$. We denote by $\delta_1(\xi,t)>0$ the greatest number such that the solution $\alpha$ is defined on $[0,\delta_1[$.
\begin{lem}\label{lem:estimate_phi_on_curve}
Let $t \in W_0 \setminus \Sigma_\xi$. Then for every $0 \leq \tau < \delta_1(\xi,t)$ the following estimate holds:
\begin{equation}\label{eq:estimate_phi_on_curve_1}
    \phi(t) \cdot \xi - \phi(\alpha_{\xi,t}(\tau)) \cdot \xi \geq \big( (1-\theta)(2C_L)^{-1} \big)^{\frac{1}{1-\theta}} \tau^{\frac{1}{1-\theta} }.
\end{equation}
\end{lem}
Before proving this lemma we shall recall some useful inequalities:
\begin{lem}\label{lem:elementar_inequality}
If $0 < a \leq b$ and $1 < r $ then 
\begin{align*}
	& (b - a)^r \leq b^r - a^r,\\
	& (a + b)^r \leq 2^r (a^r + b^r).
\end{align*}
\end{lem}
\begin{proof}[Proof of \ref{lem:estimate_phi_on_curve}]
Let $t \in W_0 \setminus \Sigma_\xi$ be fixed. Consider the function $h(\tau)=\phi(\alpha(\tau)) \cdot \xi$, for $0 \leq \tau < \delta_1$. We first note that $ h^\prime(\tau) = -\|  {}^\mathrm{t}\mathrm{d} \phi (\alpha(\tau)) \xi \|$, so $h(\tau)$ is a (strictly) decreasing function. Now set $H(\tau)=|h(\tau)|^{1-\theta}$. If $h(0) \leq 0$ then $H^\prime(\tau) \geq (1-\theta)C_L^{-1}$. Integrating in $\tau$ and using the definition of $H(\tau)$ one obtain that
\begin{equation*}
	\big(-\phi(\alpha(\tau)) \cdot \xi \big)^{1-\theta} -\big(-\phi(t) \cdot \xi \big)^{1-\theta} \geq (1-\theta) C_L^{-1} \tau.
\end{equation*}
Thus 
\begin{align*}
	\big((1-\theta) C_L^{-1} \tau \big)^{\frac{1}{1-\theta}} & \leq \bigg( \big(-\phi(\alpha(\tau)) \cdot \xi \big)^{1-\theta} - \big( -\phi(t) \cdot \xi \big)^{1-\theta} \bigg)^{\frac{1}{1-\theta}}\\
	& \leq \phi(t) \cdot \xi - \phi(\alpha(\tau)) \cdot \xi,
\end{align*}
for $0 \leq - \phi(t) \cdot \xi \leq - \phi(\alpha(\tau)) \cdot \xi$, and $\frac{1}{1-\theta} > 1$. Now if $h(0) > 0$ then we shall analyze two cases:
\begin{enumerate}
	\item If $h(\tau)>0$, for every $\tau$, using the same argument as before one obtain that
	\begin{equation*}
		 \big((1-\theta) C_L^{-1} \tau \big)^{\frac{1}{1-\theta}} \leq \phi(t) \cdot \xi - \phi(\alpha(\tau)) \cdot \xi
	\end{equation*}
	\item If there exists a $\tau_0\in [0,\delta_1[$ such that $h(\tau_0)=0$ then by we already know that estimate holds for $0 \leq \tau \leq \tau_0$. So it is enough to consider $\tau > \tau_0$. We have that
          \begin{align*}
              \phi(t) \cdot \xi - \phi(\alpha(\tau)) \cdot \xi &= \phi(t) \cdot \xi - \phi(\alpha(\tau_0)) \cdot \xi + \phi(\alpha(\tau_0)) \cdot \xi - \phi(\alpha(\tau)) \cdot \xi\\
              &\geq \big((1-\theta)C_L^{-1}\big)^{\frac{1}{1-\theta}}\tau_0^{\frac{1}{1-\theta}}+\phi(\alpha_{\xi,\alpha_{\xi,t}(\tau_0)}(0))\cdot\xi-\phi(\alpha_{\xi,\alpha_{t}(\tau_0)}(\tau-\tau_0))\cdot\xi\\
              &\geq \big((1-\theta)C_L^{-1}\big)^{\frac{1}{1-\theta}}\tau_0^{\frac{1}{1-\theta}} + \big((1-\theta)C_L^{-1}\big)^{\frac{1}{1-\theta}}\big(\tau-\tau_0\big)^{\frac{1}{1-\theta}}\\
              &=\big((1-\theta)C_L^{-1}\big)^{\frac{1}{1-\theta}}\bigg(\tau_0^{\frac{1}{1-\theta}}+\big(\tau-\tau_0\big)^{\frac{1}{1-\theta}}\bigg)\\
              &\geq \big((1-\theta)(2C_L)^{-1}\big)^{\frac{1}{1-\theta}}\tau^{\frac{1}{1-\theta}},
          \end{align*}
          where we used that $\alpha_{\xi,t}(\tau)$ is a 1-parameter subgroup.
      \end{enumerate}
\end{proof}
Now we shall construct the curves $\gamma_{\xi,t}(\tau)$ for every $\xi \in \Gamma \cap \mathbb{S}^{m-1}$, and $t \in W_0 \setminus \Sigma_\xi$.  One consequence of Lemma \ref{lem:estimate_phi_on_curve} is that $\sup_{t \in W \setminus \Sigma_\xi} \delta_1(\xi,t) < \infty$, and therefore the limit 
\begin{equation*}
	\ell(\xi,t)\doteq\lim_{\tau\to\delta_1(\xi,t)}\alpha_{\xi,t}(\tau),
\end{equation*}
exists and belong to the boundary of $W_0 \setminus \Sigma_\xi$. We shall split the construction of the curve $\gamma_{\xi,t}(\tau)$ in two cases:
\begin{itemize}
	\item $\ell(x,t) \in \partial W_0$; \\
		In this cases we define $\gamma_{\xi,t} : [0,\delta_1] \longrightarrow \overline{W_0}$, by
		\begin{equation*}
   			\gamma_{\xi,t}(\tau)=
			\begin{cases}
        				\alpha_{\xi,t}(\tau),& \quad 0 \leq \tau < \delta_1(\xi,t);\\
        				\ell(\xi,t), & \quad \tau = \delta_1(\xi,t).
        			\end{cases}
   	 	\end{equation*}
	\item $\ell(\xi,t) \in \Sigma_\xi \setminus \partial W_0;$\\
	By Property $(\star)$ we have that $\phi(\ell(\xi,t)) \cdot \xi=0$, so by Lemma \ref{lem:estimate_phi_on_curve}, we have that 
	\begin{equation*}
		\phi(t) \cdot \xi \geq \big((1-\theta)(2C_L)^{-1}\big)^{\frac{1}{1-\theta}} \delta_1(\xi,t)^{\frac{1}{1-\theta}}.
	\end{equation*}
	Now since  the map $t \mapsto \phi(t) \cdot \xi$ is open (and continuous), there is $t_0 \in W_0$ such that
	 \begin{itemize}
 		\item  $\phi(t_0) \cdot \xi<0$;
 		\item for every $0\leq\lambda\leq 1$ we have that 
 		\begin{equation*}
			\phi(t) \cdot \xi - \phi(\lambda t_0+(1-\lambda) \ell(\xi,t)) \cdot \xi \geq \frac{1}{2} \big((1-\theta) (2C_L)^{-1} \big)^{\frac{1}{1-\theta}} \delta_1(\xi,t)^{\frac{1}{1-\theta}};
 		\end{equation*}
 		\item $\| t_0 - \ell(\xi,t)\| \leq \delta_1(\xi,t)$.
	\end{itemize}
	In this case we define 
	\begin{equation*}
   			\gamma_{\xi,t}(\tau)=
			\begin{cases}
        				\alpha_{\xi,t} \bullet [\ell(\xi,t),t_0] \bullet\alpha_{\xi,t_0} (\tau) ,&\quad 0 \leq \tau < \delta_1(\xi,t)+\|t_0-\ell(\xi,t)\|+\delta_1(\xi,t_0) ;\\
        				\ell(x,t_0) ,&\quad \tau= \delta_1(\xi,t)+\|t_0-\ell(\xi,t)\|+\delta_1(\xi,t_0),
        			\end{cases}
   	 	\end{equation*}
		where $[\ell(\xi,t),t_0]$ is the line segment with derivative constant and equals to one, and the $\bullet$ sign stand for concatenation.
\end{itemize}
We denote by $\delta_0(\xi,t)$ the upper bound of the interval where $\gamma_{\xi,t}$ is defined, that is,
\begin{equation*}
	\gamma_{\xi,t} : [0,\delta_0(\xi,t)] \longrightarrow \overline{W_0}.
\end{equation*}
We claim that the family of curves $\gamma_{\xi,t}$ satisfies the following properties:
\begin{enumerate}
	\item For every $W_1 \Subset W_0$ open ball centered at the origin there exists $\delta > 0$ such that $\delta_0(\xi,t)\geq \delta$, for every $\xi \in \Gamma \cap \mathbb{S}^{m-1}$, and $ t \in W_1\setminus \Sigma_{\xi}$;\\
	\item The length of the curves $\gamma_{\xi,t}$ is uniformly bounded for every $\xi \in \Gamma \cap \mathbb{S}^{m-1}$, and $ t \in W_0 \setminus \Sigma_{\xi}$; \\
     	\item If $\xi \in \Gamma \cap \mathbb{S}^{m-1}$, and $ 0 \leq \tau \leq \delta_0(\xi,t)$, then 
     	\begin{equation*}
     		\Phi(t) \cdot \xi - \Phi(\gamma_{\xi,t}(\tau)) \cdot \xi \geq \frac{1}{2} \big( (1-\theta) (8C_L)^{-1} \big)^{\frac{1}{1-\theta}} \tau^{\frac{1}{1-\theta}}
     	\end{equation*}
\end{enumerate}
Since $\|\gamma_{\xi,t}^\prime(\tau)\|=1$ for every $\tau$, except for eventually two points, then by Lemma \ref{lem:estimate_phi_on_curve} and the definition of $\delta_0$,
\begin{align*}
	\mathrm{length}(\gamma_{\xi,t}) & = \int_{0}^{\delta_0(\xi,t)}\|\gamma_{\xi,t}^\prime(\tau^\prime)\| \mathrm{d}\tau^\prime\\
	&= \delta_0(\xi,t)\\
	& \leq  3 \big(2 \sup_{\overline{W}} \|\phi \|\big)^{1-\theta}(1-\theta)^{-1} 2C_L.
\end{align*}
Let $r > 0$ be the radius of $W_0$ and let $ 0 < r_1 < r$. So if $|t| < r_1$ we have that
\begin{align*}
	r - r_1& \leq \|\gamma_{\xi,t}(\delta_0(\xi,t)) - t \| \\
	& \leq \mathrm{length}(\gamma_{\xi,t})\\
	& = \delta_0(\xi,t).
\end{align*}
The last property follows directly from Lemma \ref{lem:estimate_phi_on_curve} if $\ell(\xi,t) \in \partial W_0$. Now if $t$ is such that $\ell(\xi,t) \in \Sigma_\xi \setminus \partial W_0$, then we have three cases to analyze. If $0 \leq \tau \leq \delta_1(\xi,t)$, the desired inequality follows from Lemma \ref{lem:estimate_phi_on_curve}. Now  if $\delta_1(\xi,t) \leq \tau \leq \delta_1(\xi,t)+ \| t_0 + \ell(\xi,t) \| \leq 2 \delta_1(\xi,t)$, then $\gamma_{\xi,t}(\tau) = \lambda t_0+(1-\lambda) \ell(\xi,t)$ for some $0 \leq \lambda \leq 1$, so by the definition of $t_0$,
\begin{align*}
	\phi(t) \cdot \xi - \phi(\gamma_{\xi,t}(\tau)) \cdot \xi &= \phi(t) \cdot \xi - \phi(\lambda t_0+(1-\lambda) \ell(\xi,t)) \cdot \xi\\
	&\geq \frac{1}{2}\big((1-\theta)(2C_L)^{-1}\big)^{\frac{1}{1-\theta}}\delta_1(\xi,t)^{\frac{1}{1-\theta}}\\
	&\geq\frac{1}{2}\big((1-\theta)(2C_L)^{-1}\big)^{\frac{1}{1-\theta}}\left(\frac{\tau}{2}\right)^{\frac{1}{1-\theta}}\\
	& = \frac{1}{2}\big((1-\theta)(4C_L)^{-1}\big)^{\frac{1}{1-\theta}}\tau^{\frac{1}{1-\theta}}.
\end{align*}
If $\delta_1(\xi,t)+ \| t_0 + \ell(\xi,t) \|  \leq \tau \leq \delta_1(\xi,t)+ \| t_0 + \ell(\xi,t) \| + \delta_1(\xi,t_0)$, then
\begin{align*}
              \phi(t) \cdot \xi - \phi(\gamma_{\xi,t}(\tau)) \cdot \xi &= \Phi(t) \cdot \xi - \phi(t_0) \cdot \xi + \phi(t_0) \cdot \xi - \phi(\gamma_{\xi,t}(\tau)) \cdot \xi\\
              &\geq \frac{1}{2}\big((1-\theta)(4C_L)^{-1}\big)^{\frac{1}{1-\theta}}(\delta_1(\xi,t)+ \| t_0 + \ell(\xi,t) \|)^{\frac{1}{1-\theta}} \\
              &+ \big((1-\theta)(2C_L)^{-1}\big)^{\frac{1}{1-\theta}}(\tau - \delta_1(x,t) - \| t_0 - \ell(\xi,t) \|)^{\frac{1}{1-\theta}} \\
              &\geq \frac{1}{2}\big((1-\theta)(4C_L)^{-1}\big)^{\frac{1}{1-\theta}}\bigg((\delta_1(\xi,t)+ \| t_0 + \ell(\xi,t) \|)^{\frac{1}{1-\theta}}\\
              & + (\tau - \delta_1(\xi,t) - \| t_0 - \ell(\xi,t) \|)^{\frac{1}{1-\theta}}\bigg)\\
              &\geq \frac{1}{2}\big((1-\theta)(8C_L)^{-1}\big)^{\frac{1}{1-\theta}} \tau^{\frac{1}{1-\theta}},
\end{align*}
where in the last inequality we used Lemma \ref{lem:elementar_inequality}.
\end{proof}

%
\comment{
\begin{thm}\label{thm:microlocal_smooth_hypoelliptic_tube}
Let $\xi_0 \in \mathbb{S}^{m-1}$ be such that the map $\Phi$ satisfies condition $(\star)$ at $\xi_0$, with $\theta = 1/2$. Let  $u \in \mathcal{C}^1 (W; \mathcal{D}^\prime(V))$ be such that $(0,0,\xi_0,0) \notin \mathrm{WF}(\mathrm{L}_j u)$, for $j = 1, \dots, n$. Then $(0,0,\xi_0,0) \notin \mathrm{WF}(u)$.
\end{thm}
}
%
%
%
\comment{
The corollaries follows from the microlocal versions for by the elliptic regularity Theorem, which says that $\mathrm{WF}(u) \subset \mathrm{Char}(\mathcal{V}) \cup \mathrm{WF}(\mathbb{L}u)$ and $\mathrm{WF}_s(u)  \subset \mathrm{Char}(\mathcal{V}) \cup \mathrm{WF}_s(\mathbb{L}u)$, where $\mathbb{L}u = \sum_{j=1}^n \mathrm{L}_j u \mathrm{d}t_j$, so the only points we need to care about are the ones in $\mathrm{Char}(\mathcal{V})$, and if $u \in  \mathcal{D}^\prime(U)$ is such that $\mathrm{L}_j u \in \mathcal{C}^\infty(U)$, then by the usual elliptic theory (see Proposition I.4.3 of \cite{trevesbook}), we have that there exists $V_0 \subset V$, and $W_0 \subset W$ open balls centered at the origin such that $V_0 \times W_0 \Subset U$, and $u|_{V_0 \times W_0} \in \mathcal{C}^\infty (W_0; \mathcal{D}^\prime(V_0))$.
}
%
%
\begin{proof}[Proof of Therem \ref{thm:microlocal_hypoelliptic_tube}]
Let $V_0 \subset V$ and $W_0 \subset W$ be open neighborhoods of the origin, $\Gamma \subset \mathbb{R}^m \setminus 0$ be an open cone containing $\xi_0$ and $C_L, \theta$ be positive constants with $ 1/2 \leq \theta < 1$, such that 
\begin{itemize}
	\item for every $\xi \in \Gamma$ the map $W_0 \ni t \mapsto \phi(t) \cdot \xi$ is open;
	\item for every $(t,\xi) \in W_0 \times (\Gamma \cap \mathbb{S}^{m-1}$, $| \phi(t) \cdot \xi |^\theta \leq C_L \| {}^{\mathrm{t}} \mathrm{d} \phi(t) \xi \|$
\end{itemize}
Let $W_1 \Subset W_0$ be an open neighborhood of the origin, and let $V_2 \Subset V_1 \Subset V_0$ be open neighborhoods of the origin. Consider $\chi \in \mathcal{C}^\infty_c(V_0)$ a cut-off function such that $\chi \equiv 1$ on $V_1$. For every $\xi \in \Gamma$ we denote $\dot{\xi} =\frac{\xi}{|\xi|}$. Now let $0 < \kappa^{-1} \leq s$ to be fixed later, $\xi \in \Gamma$ with $|\xi| \geq 1$, $t \in W_1 \setminus \Sigma_{\dot{\xi}}$ and $x \in V_2$. Writing $\gamma = \gamma_{\dot{\xi},t}$, $\delta_0 = \delta_0(\dot{\xi},t)$, and $z = Z(x,t)$, we have that
\begin{align}\label{eq:fundamental_theorem_calculus_FBI}
	\mathfrak{F}_\kappa [\chi u](t_\ast ; z,\xi) - \mathfrak{F}_\kappa [\chi u](t ; z,\xi) &= \int_0^{\delta_0} \frac{\mathrm{d}}{\mathrm{d} \tau} \left( \mathfrak{F}_\kappa [\chi u](\gamma(\tau) ; z,\xi) \right) \mathrm{d} \tau \\ \nonumber
	& = \sum_{j=1}^n \int_0^{\delta_0}  \mathfrak{F}_\kappa [\chi \mathrm{L}_j u](\gamma(\tau);z,\xi)\gamma_j^\prime(\tau)\mathrm{d}\tau\\ \nonumber
	& +\sum_{j=1}^n \int_0^{\delta_0}  \mathfrak{F}_\kappa [u\mathrm{L}_j \chi ](\gamma(\tau);z,\xi) \gamma_j^\prime(\tau)\mathrm{d}\tau.
\end{align}
Our goal is to estimate $| \mathfrak{F}_\kappa [\chi u](t; z, \xi)|$, and to do so we shall use \eqref{eq:fundamental_theorem_calculus_FBI} and estimate all the other terms. By Proposition \ref{prop:curves}, if $t^\prime = \gamma(\tau)$ for some $ 0 \leq \tau \leq \delta_0$ then
\begin{equation*}
	\phi(t) \cdot \xi - \phi(t^\prime) \cdot \xi \geq c \tau^{\frac{1}{1-\theta}} |\xi|,
\end{equation*}
where $c = \frac{1}{2}\big((1-\theta)(8C_L)^{-1}\big)^{\frac{1}{1-\theta}}$. Let us start by the most involving term, \textit{i.e.} estimating $|\mathfrak{F}_\kappa [\chi \mathrm{L}_j u](\gamma(\tau);z,\xi)| $. Let $t^\prime = \gamma(\tau)$, for some $0 \leq \tau \leq \delta_0$, and let $j \in \{1, \dots, n \}$ be fixed. As in the first part of the proof of Theorem \ref{thm:microlocal_Gevrey_FBI} we shall assume for simplicity that $\chi \mathrm{L}_j u = \mathrm{bv}(F)$, where $F \in \mathcal{C}^\infty (\mathcal{W}_\delta(V_0 \times W_0, \widetilde{\Gamma}))$, $\widetilde{\Gamma} \subset \mathbb{R}^{n+m} \setminus 0$ is an open convex cone with $(\xi_0, 0) \cdot \widetilde{\Gamma} <0$, and 
\begin{equation*}
 	\begin{cases}
		|F((x,t) + iv)| \leq \frac{C_1}{|v|^{k_0}};\\
		| (\bar{\partial}_z + \bar{\partial_\eta}) F((x,t) + iv)| \leq C_F^{k+1} k!^{s-1}|v|^k
	\end{cases}
	\eqno{
		\begin{matrix}
			x \in V_1 & t \in W_1 \\
			k \in \mathbb{Z}_+ & v \in \widetilde{\Gamma}_\delta.
		\end{matrix}
		}
\end{equation*}
Fixing $v = (v^\prime, v^{\prime \prime}) \in \widetilde{\Gamma}$ and writing $Q(w,w^\prime,\zeta) = i \zeta \cdot (w - w^\prime) - \langle \zeta \rangle^\kappa \langle w - w^\prime \rangle^2$, we can write 
\begin{equation*}
 	\mathfrak{F}_\kappa[\chi \mathrm{L}_j u](t^\prime; z, \xi) = \lim_{\lambda \to 0^+} \int_{V} e^{Q(z, Z(x^\prime, t^\prime), \xi)} \chi(x^\prime) f(x^\prime + i\lambda v^\prime, t^\prime + i\lambda v^{\prime \prime}) \Delta_\kappa(z - Z(x^\prime, t^\prime), \xi) \mathrm{d} x^\prime.  
\end{equation*}
Let $0 < \lambda < \delta$ . Then by Stokes' Theorem, for  $0 < \sigma \ll 1$, 
\begin{align*}
 	\int_{V} & e^{Q(z, Z(x^\prime, t^\prime), \xi)} \chi(x^\prime) f(x^\prime + i\lambda v^\prime, t^\prime + i\lambda v^{\prime \prime}) \Delta_\kappa(z - Z(x^\prime, t^\prime), \xi) \mathrm{d} x^\prime = \\
	& = \int_{V_1} e^{Q(z, Z(x^\prime, t^\prime), \xi)} F(x^\prime + i\lambda v^\prime, t^\prime + i\lambda v^{\prime \prime}) \Delta_\kappa (z - Z(x^\prime, t^\prime), \xi) \mathrm{d} x^\prime\\	
  	& + \int_{V \setminus V_1} e^{Q(z, Z(x^\prime, t^\prime), \xi)} \chi(x^\prime)f(x^\prime + i\lambda v^\prime, t^\prime + i\lambda v^{\prime \prime}) \Delta_\kappa (z - Z(x^\prime, t^\prime), \xi) \mathrm{d} x^\prime\\
	& = \int_{V_1} e^{Q(z, Z(x^\prime, t^\prime) + i \sigma v^\prime, \xi)} F(x^\prime + i (\sigma + \lambda)v^\prime, t^\prime + i \lambda v^{\prime \prime}) \Delta_\kappa (z - Z(x^\prime, t^\prime) - i\sigma v^\prime, \xi) \mathrm{d} x^\prime\\
	& + \int_{\partial V_1 + i [0, \sigma]v^\prime} e^{Q(z, z^\prime + i \phi(t^\prime), \xi)} F(z^\prime + i \lambda v^\prime, t^\prime + i \lambda v^{\prime \prime}) \Delta_\kappa (z - z^\prime - i\phi(t^\prime), \xi) \mathrm{d} z^\prime\\
	& - \int_{V_1 + i [0, \sigma] v^\prime}e^{Q(z, z^\prime + i \phi(t^\prime), \xi)}\sum_{k=1}^m \bar{\partial}_z F(z^\prime + i \lambda v^\prime, t^\prime + i\lambda v^{\prime \prime})\Delta_\kappa (z - z^\prime - i\phi(t^\prime), \xi) \mathrm{d} \bar{z^\prime} \wedge \mathrm{d} z^\prime \\
	& + \int_{V \setminus V_1} e^{Q(z, Z(x^\prime, t^\prime), \xi)} \chi(x^\prime) F(x^\prime + i \lambda v^\prime, t^\prime + i \lambda v^{\prime \prime}) \Delta_\kappa (z - Z(x^\prime, t^\prime), \xi) \mathrm{d} x^\prime.
\end{align*}
To estimate the exponential in all the four integrals we shall use the precise control of $\phi(\cdot) \cdot \xi$ along the curve $\gamma$ and the fact that $v^\prime \cdot \xi_0 < 0$. Shrinking $\Gamma$ if necessary, there exist a positive constant $c^\sharp > 0$ such that $v^\prime \cdot \zeta \leq -c^\sharp |\zeta|$, for all $\zeta \in \Gamma$. Writing $z^\prime = Z(x^\prime, t^\prime) + i \sigma^\prime v^\prime$, where $0 \leq \sigma^\prime \leq \sigma$, we have that 
\begin{align*}
	\left| e^{Q(z, z^\prime,\xi)} \right| & = e^{ - \big( \phi(t) \cdot \xi - \phi(t^\prime) \cdot \xi \big) + \sigma^\prime v^\prime \cdot \xi - |\xi|^\kappa \left(|x-x^\prime|^2 - \left|\phi(t) - \phi(t^\prime) - \sigma^\prime v^\prime \right|^2\right) }\\ 
	&\leq e^{ - c \tau^{\frac{1}{1-\theta}} |\xi| - \sigma^\prime c^\sharp |\xi| - |x - x^\prime|^2 |\xi|^\kappa + 2 {\sigma^\prime}^2 |v^\prime|^2 |\xi|^\kappa + 2 |\phi(t) - \phi(t^\prime)|^2 |\xi|^\kappa } \\ 
	& \leq e^{ - |x - x^\prime|^2 |\xi|^\kappa - \sigma^\prime |\xi|\left(c^\sharp - 2|v^\prime|^2|\xi|^{\kappa - 1}\right) - |\xi|^\kappa \left[ c\tau^{\frac{1}{1-\theta}}|\xi|^{1-\kappa} - 2C_\phi^2 |t - t^\prime|^2 \right] } \\
	& \leq e^{ - |x - x^\prime|^2 |\xi|^\kappa - \sigma^\prime \frac{c^\sharp}{2}|\xi| - |\xi|^\kappa \left[ c\tau^{\frac{1}{1-\theta}}|\xi|^{1-\kappa} - 2C_\phi^2\tau^2 \right] },
\end{align*}
if we assume $\sigma < c^\sharp/4|v^\prime|$, where $C_\phi$ is the Lipchitz constant of $\phi$, and we used that $|t-t^\prime|$ can be bounded by the length of the curve $\gamma$ from $0$ to $\tau$, which is $\tau$, and $|\xi| \geq 1$. Since the term $e^{-|\xi|^\kappa \left[ c\tau^{\frac{1}{1-\theta}}|\xi|^{1-\kappa} - 4C_\phi^2\tau^2 \right] }$ does not depend on $x^\prime$, $\sigma^\prime$ or $\lambda$, it can be put outside of the integrals and the limit. Carrying out the same argument as in the \textit{1.} $\Rightarrow$ \textit{2.} of the proof of Theorem \ref{thm:microlocal_Gevrey_FBI} we obtain
\begin{equation*}
	|\mathfrak{F}_\kappa [\chi \mathrm{L}_j u](t^\prime; z,\xi) | \leq \text{Const} \cdot e^{-\epsilon |\xi|^{\frac{1}{s}} - |\xi|^\kappa \left[ c\tau^{\frac{1}{1-\theta}}|\xi|^{1-\kappa} - 4C_\Phi^2\tau^2 \right]  },
\end{equation*}
for some $\epsilon > 0$. We shall prove that the term $e^{-|\xi|^\kappa \left[ c\tau^{\frac{1}{1-\theta}}|\xi|^{1-\kappa} - 4C_\phi^2\tau^2 \right] }$ is "harmless" when compared to $e^{-\epsilon |\xi|^{\frac{1}{s}}}$ (or to any polynomial decay in the case of $\theta = 1/2$). To do so, we calculate the minimum of the function $g(\tau) = c |\xi| \tau^{\frac{1}{1-\theta}} - c_1|\xi|^\kappa \tau^2$, where $c_1 = 4C_\phi^2$. Since $g^\prime(\tau) = \frac{c|\xi|}{1-\theta} \tau^{\frac{1}{1-\theta}-1} - 2c_1 |\xi|^\kappa\tau$, if $\theta > 1/2$, we have that $g^\prime(\tau_0) = 0$ if and only if
\begin{equation*}
	\tau_0 = \left(\frac{2 c_1}{c}(1-\theta)|\xi|^{\kappa-1}\right)^{\frac{1-\theta}{2\theta-1}}.
\end{equation*}
Thus the minimum of $g(\tau)$ is
\begin{align*}
	g(\tau_0)&= c_1|\xi| \left( \frac{2 c_1}{c} (1-\theta) |\xi|^{\kappa-1}\right)^{\frac{1}{2\theta-1}} - c_1 |\xi|^\kappa \left( \frac{2 c_1}{c} (1-\theta) |\xi|^{\kappa-1} \right)^{\frac{2-2\theta}{2\theta-1}}\\
	&=|\xi|^{\frac{ 2\theta + \kappa - 2}{2\theta-1}} \left\{ c_1 \left( \frac{2 c_1}{c} (1-\theta) \right)^{\frac{1}{2\theta-1}} - c_1 \left( \frac{2 c_1}{c} (1-\theta) \right)^{\frac{2-2\theta}{2\theta-1}} \right\}
\end{align*}
Now $\frac{2\theta+\kappa-2}{2\theta-1} < \frac{1}{s}$ is equivalent to $\kappa <\frac{2\theta-1}{s}+2-2\theta$, and since $\frac{1}{s} < \frac{2\theta-1}{s}+2-2\theta$, it is aways possible to find a $\kappa$ satisfying $\frac{1}{s} < \kappa < \frac{2\theta-1}{s}+2-2\theta$. If $\theta = 1/2$, then $g(\tau) = \tau^2(c|\xi| - c_1|\xi|^\kappa)$, thus $g(\tau) \geq 0$ if $|\xi| \geq (c_1/c)^{1/(1-\kappa)}$, therefore the extra exponential is trivially harmless, \textit{i.e.} it "loses" even to polynomial decay. Thus we have obtained that
\begin{equation*}
	|\mathfrak{F}_\kappa [\chi \mathrm{L}_j u](t^\prime; z,\xi) | \leq \text{Const} \cdot e^{-\epsilon |\xi|^{\frac{1}{s}}},
\end{equation*}
for some positive constant $\epsilon > 0$. The estimate of $|\mathfrak{F}_\kappa [u \mathrm{L}_j \chi] (t^\prime; z,\xi)|$ follows from the fact that $x \in V_2 \Subset V_1$, and $x^\prime \notin V_1$, so
\begin{equation*}
	|\mathfrak{F}_\kappa [u \mathrm{L}_j \chi] (t^\prime; z,\xi)| \leq \text{Const} \cdot e^{- \epsilon |\xi|^{\kappa} },
\end{equation*}
for some $\epsilon > 0$. The last partial F.B.I left to estimate in \eqref{eq:fundamental_theorem_calculus_FBI} is  $|\mathfrak{F}_\kappa [\chi u] (t_\ast; z,\xi)|$. To obtain the desired decay it is enough to bound the exponential term in the partial F.B.I. transform, and we have already proven that
\begin{equation*}
\left| e^{Q(z,Z(x^\prime,t_\ast),\xi)} \right| \leq e^{-|x - x^\prime|^2 |\xi|^\kappa - |\xi|^\kappa \left( c \delta_0^{\frac{1}{1-\theta}} |\xi|^{1-\kappa} - 4C_\phi^2 \delta_0^2  \right)}.
\end{equation*}
Here we recall that $\delta_0 = \delta_0(\xi,t)$. By Proposition \ref{prop:curves}, $\delta_0 > \delta$, so 
\begin{align*}
	c\delta_0^{\frac{1}{1-\theta}}|\xi|^{1-\kappa} - 4C_\phi^2\delta_0^2 &= \delta_0^2 \left[ c |\xi|^{1-\kappa} \delta_0^{\frac{1}{1-\theta} - 2} - 4C_\phi^2\right] \\
	&\geq \delta_0^2 \left[ c |\xi|^{1-\kappa} \delta^{\frac{1}{1-\theta} - 2} - 4C_\phi^2\right]\\
	&\geq \delta^2 4 C_\phi^2,
\end{align*}
if $|\xi| \geq \left[ \frac{8C_\phi^2}{c \delta^{\frac{1}{1-\theta} - 2} } \right]^{\frac{1}{1-\kappa}}$. Therefore, assuming $|\xi| \geq \left[ \frac{8C_\phi^2}{c \delta^{\frac{1}{1-\theta} - 2} } \right]^{\frac{1}{1-\kappa}}$,
\begin{equation*}
	|\mathfrak{F}_\kappa [ \chi u] (t_\ast; z,\xi)| \leq \text{Const} \cdot e^{- \epsilon |\xi|^{ \kappa}},
\end{equation*}
for some $\epsilon > 0$. Summing up, we have obtained the following: There exist positive constants $C,\epsilon>0$ such that for every $x \in V_2$, $\frac{1}{s} < \kappa < \frac{2\theta-1}{s}+2-2\theta$,  $\xi \in \Gamma$ with $|\xi| \geq \left[ \frac{8C_\phi^2}{c \delta^{\frac{1}{1-\theta} - 2} } \right]^{\frac{1}{1-\kappa}}$ and $t \in W_1 \setminus \Sigma_{\dot{\xi}}$ the following estimates hold true:
\begin{equation*}
	\begin{cases}
		|\mathfrak{F}_\kappa [\chi \mathrm{L}_j u](t^\prime; z,\xi) | \leq C e^{-\epsilon |\xi|^{\frac{1}{s}}};\\
		|\mathfrak{F}_\kappa [u \mathrm{L}_j \chi] (t^\prime; z,\xi)| \leq C e^{- \epsilon |\xi|^{\frac{1}{s}}},\\
		|\mathfrak{F}_\kappa [ \chi u] (t_\ast; z,\xi)| \leq C e^{- \epsilon |\xi|^{\frac{1}{s}}}.
	\end{cases}
\end{equation*}
Since the length of the curves $\gamma_{\dot{\xi},t}$ are uniformly bounded for $\xi \in \Gamma$ and $t \in W_2 \setminus \Sigma_{\dot{\xi}}$, we have that, for every $x \in V_2$, $\frac{1}{s} < \kappa < \frac{2\theta-1}{s}+2-2\theta$,  $\xi \in \Gamma$ with $|\xi| \geq \left[ \frac{8C_\phi^2}{c \delta^{\frac{1}{1-\theta} - 2} } \right]^{\frac{1}{1-\kappa}}$ and $t \in W_1 \setminus \Sigma_{\dot{\xi}}$ 
\begin{equation*}
	|\mathfrak{F}_\kappa [\chi u](t; Z(x,t), \xi)| \leq \text{Const} \cdot e^{-\epsilon |\xi|^\frac{1}{s}}.
\end{equation*}
In view of $\Sigma_{\dot{\xi}}$ being an analytic variety for every $\xi \in \Gamma$, and the set $|\xi| \leq \left[ \frac{8C_\phi^2}{c \delta^{\frac{1}{1-\theta} - 2} } \right]^{\frac{1}{1-\kappa}}$ being compact, if $\frac{1}{s} < \kappa < \frac{2\theta-1}{s}+2-2\theta$, there exist $C, \epsilon > 0$ such that
\begin{equation*}
	|\mathfrak{F}_\kappa [\chi u](t; Z(x,t), \xi)| \leq C e^{-\epsilon |\xi|^\frac{1}{s}},
\end{equation*}
for every $x \in V_2$, $t \in W_1$, and $\xi \in \Gamma$. Therefore, by Theorem \ref{thm:microlocal_Gevrey_FBI_tube}, we have that $(0,0,\xi_0,0) \notin \mathrm{WF}_s(u)$. Now if $\theta = 1/2$, and $(0,0,\xi_0,0) \notin \text{WF}(\mathbb{L}u )$, all estimates but one remains the same, and the one that changes is $|\mathfrak{F}_\kappa [\chi \mathrm{L}_j u](t^\prime; z,\xi)|$. As we already pointed out, when $\theta = 1/2$, the term $e^{-|\xi|^\kappa \left[ c\tau^2 \xi|^{1-\kappa} - 4C_\phi^2\tau^2 \right] }$ is "harmless" when compared with polynomial decay, which is exactly the type of decay that will replace $e^{-\epsilon|\xi|^{\frac{1}{s}}}$.
\end{proof}
\begin{cor}\label{cor:Gevrey_hypoelliptic_tube}
Let $s > 1$. If the map $\phi$ satisfies condition $(\star)$ at all points $\xi_0 \in \mathbb{S}^{m-1}$ then the system $\mathcal{V}$ is $s$-Gevrey hypoelliptic at the origin.
\end{cor}
\begin{cor}\label{cor:Gevrey_hypoelliptic_tube}
If the map $\phi$ satisfies condition $(\star)$ at all points $\xi_0 \in \mathbb{S}^{m-1}$, with $\theta = 1/2$, then the system $\mathcal{V}$ is hypoelliptic at the origin.
\end{cor}
Since the proofs of these corollaries are similar, we shall only prove Corollary \ref{cor:Gevrey_hypoelliptic_tube}.
\begin{proof}[Proof of Corollary \ref{cor:Gevrey_hypoelliptic_tube}]
By the elliptic regularity Theorem we have that $\text{WF}_s(u) \big|_0 \subset \text{Char}(\mathcal{V}) \big|_0$, for $0 \notin\text{singsupp}_s\mathbb{L} u$. The characteristic set at the origin consists only of points of the form $(0,0,\xi_0,0)$, with $\xi_0 \in \mathbb{R}^m \setminus \{0\}$, for we are assuming $\mathrm{d}\phi(0) = 0$. Thus by Theorem \ref{thm:microlocal_hypoelliptic_tube}, $(0,0,\xi_0,0) \notin \text{WF}_s(u)$, since the map $\phi$ satisfies condition $(\star)$ at $\xi_0$, for every $\xi_0 \in \mathbb{S}^{m-1}$\textit{i.e.} $\text{WF}_s(u)\big|_0 = \emptyset$, for $\text{WF}_s(u)$ is conic.
\end{proof}


\subsection{Maire's example}

%
In $\mathbb{R}^4$ consider the following complex vector fields:
\begin{align*}
	&\mathrm{L}_1 = \partial_{t_1} + i\left( 3 \partial_{x_1} - (4 t_1 t_2  + 3)t_1^2 \partial_{x_2} \right),\\
	&\mathrm{L}_2 = \partial_{t_2} -i t_1^4 \partial_{x_2}.
\end{align*}
We already discussed that the system $\{ \mathrm{L}_1, \mathrm{L}_2 \}$ is not Gevrey-$s$ hypoelliptic for every $s \geq 4$, therefore it cannot satisfy property $(\star)$ at all $\xi \in \mathbb{S}^1$. In this section we shall prove this fact explicitly, to be more precise, we shall prove that the Lojasiewicz inequality with parameters with the exponent $1/2 \leq \theta < 1$ fails. The first integrals associated with these complex vector fields are
\begin{align*}
	&Z_1(x,t) = x_1 - i 3 t_1, \\
	&Z_2(x,t) = x_2 + i (t_1 t_2 + 1) t_1^3.
\end{align*}
The characteristic set of the system $\{ \mathrm{L}_1, \mathrm{L}_2 \}$ is given by $\eta_1 = \eta_2 = t_1 = \xi_1 = 0$, where $(\xi, \eta)$ are the dual variables of $(x, t)$. So we only have to check property $(\star)$ at $(0,\pm 1)$. Now note that 
\begin{align*}
	& \| {}^{\mathrm{t}} \mathrm{d} \Phi(t) \xi \|^2 = t_1^8 \xi_2^2 + \left( - 3 \xi_1 + t_1^2 \xi_2 (3+ 4 t_1 t_2)\right)^2,\\
	& | \Phi(t) \cdot \xi |^2 = 9 t_1^2 \xi_1^2 + t_1^6 ( 1+ t_1 t_2)^2 \xi_2^2.
\end{align*}
So let $\xi_1^2 + \xi_2^2 = 1$. If we impose $- 3 \xi_1 + t_1^2 \xi_2 (3+ 4 t_1 t_2) = 0$, we obtain that
\begin{equation*}
	\xi_1^2 = \frac{t_1^4( 3 + 4 t_1 t_2)^2}{9 + t_1^4 (3 + 4 t_1 t_2)^2}. 
\end{equation*}
If we now set $t_2 = 0$, $\xi_1 = \frac{t_1^2}{\sqrt{1 + t_1^4}}$, and $\xi_2 = \sqrt{1 - \frac{t_1^4}{1+ t_1^4}}$, then
\begin{align*}
	& \| {}^{\mathrm{t}} \mathrm{d} \Phi(t) \xi \|^2 = t_1^8 \left(1- \frac{t_1^4}{1 + t_1^4} \right) + 9 t_1^4 \left( -  \frac{1}{\sqrt{1 + t_1^4}} + \sqrt{1 - \frac{t_1^4}{1+ t_1^4}} \right)^2,\\
	& | \Phi(t) \cdot \xi |^2 = 9 \frac{t_1^6}{1+t_1^4}  + t_1^6 \left(1 - \frac{t_1^4}{1+ t_1^4} \right).
\end{align*}
Now as $t_1$ goes to $0$ we have that $ \| {}^{\mathrm{t}} \mathrm{d} \Phi(t) \xi \| \approx t_1^4$ and $ | \Phi(t) \cdot \xi | \approx t_1^3$, therefore the exponent $\theta$ must be at least $4/3$, which is greater than $1$. Thus the system $\{ \mathrm{L}_1, \mathrm{L}_2 \}$ does not satisfy property $(\star)$ at $(0,1)$. Analogously we have that the same is valid at $(0,-1)$.
%


\subsection{A class of examples}\label{sec:examples}

%

Let $k \in \mathbb{N}$, and let $P \in \mathbb{R}[t_1, \dots, t_n]$ be a homogeneous polynomial of degree $k$. 
\begin{lem}\label{lem:gradient_polynomial}
Let $\mathcal{A} \doteq \{t \in \mathbb{R}^n \; : \; |\nabla P(t) | \leq 1 \}$. Suppose that there exists a constant $C>0$ such that
\begin{equation*}
	\sup_{t \in \mathcal{A}} |P(t)| \leq C.
\end{equation*}
Then, the following inequatilly holds true:
\begin{equation}\label{eq:lojasiewicz_homogeneous_polynomial}
	|P(t)|^{1-\frac{1}{k}} \leq C^{1-\frac{1}{k}} |\nabla P(t)|, \qquad \forall t \in \mathbb{R}^n.
\end{equation}
\end{lem}
Before proving this Lemma, we point out that this is actually an equivalence, and it is related to the problem of finding the explicit Lojasiewicz exponent for polynomials at infinity, see for instance \cite{ploski,haraux}. 
\begin{proof}
We only need to check \eqref{eq:lojasiewicz_homogeneous_polynomial} for $t \notin \mathcal{A}$. So let $t \in \mathbb{R}^n$ be such that $|\nabla P(t)| > 1$, and set $\lambda  = |\nabla P(t)|^{\frac{-1}{k-1}}$. Since $\nabla P$ is homogeneous with degree $k-1$, we have that 
\begin{align*}
	|\nabla P(\lambda t)| &= \lambda^{k-1}|\nabla P(t)| \\
	& = \frac{1}{|\nabla P(t)|} |\nabla P(t)|\\
	& = 1,
\end{align*}
\textit{i.e.} $\lambda t \in \mathcal{A}$. Thus, $|P(\lambda t)| \leq C$, and in view of the homogeneity of $P$, we have that
\begin{align*}
	|P(t)|^{1-\frac{1}{k}} & = \frac{1}{\lambda^{k-1}}|P(\lambda t)|^{1-\frac{1}{k}} \\
	& \leq C^{1-\frac{1}{k}}|\nabla P(t)|.
\end{align*}
\end{proof}
We shall use Lemma \ref{lem:gradient_polynomial} to construct a class of examples satisfying condition $(\star)$. So let $k \in \mathbb{N}$, and let $\phi(t) = (\phi_1(t), \dots, \phi_m(t))$ be such that each $\phi_j \in \mathbb{R}[t_1, \dots, t_n]$ is a homogeneous of degree $k$. For each $\xi \in \mathbb{S}^{m-1}$ we define $P_\xi(t) \doteq \phi(t) \cdot \xi$. Then $P_\xi$ is a homogeneous polynomial of degree $k$. So $\phi$ satisfies condition $(\star)$ with $\theta = (k-1)/k$ at a point $\xi_0 \in \mathbb{S}^{m-1}$ if the origin is not a local minimum of $P_{\xi_0}$ and $P_\xi$ is uniformly bounded in $\{ t \in \mathbb{R}^n \; : \; |\nabla P_\xi(t)| \leq 1\}$, for every $\xi$ in some open neighborhood of $\xi_0$. This boundedness condition is valid if, for instance, the set $\{t \in \mathbb{R}^n \; : \; |\nabla P_\xi(t) | \leq 1\}$ is bounded for every $\xi$ in some open neighborhood of $\xi_0$.
\begin{exa}
Consider in $\mathbb{R}^2$ two polynomials $\phi_1$ and $\phi_2$, homogeneous of degree $k$. Suppose that, for some $0 < \rho < 1$, 
\begin{equation*}
	\rho \|\nabla \phi_2(t)\|_1 \leq \|\nabla \phi_1(t)\|_1, \qquad \forall t \in \mathbb{R}^2.
\end{equation*}
Then, if $0 < \xi_2 < \rho^\prime \xi_1$, for any $0 < \rho^\prime < \rho$,  we have that
\begin{align*}
	|\nabla_t P_\xi |^2 & = |\xi_1 \partial_{t_1}\phi_1 + \xi_2 \partial_{t_1}\phi_2|^2 + |\xi_1 \partial_{t_2}\phi_1 + \xi_2 \partial_{t_2}  \phi_2|^2,
\end{align*}
and
\begin{align*}
	|\xi_1 \partial_{t_1}\phi_1 + \xi_2 \partial_{t_1}\phi_2| & \geq \xi_1 |\partial_{t_1}\phi_1| - \xi_2 |\partial_{t_1} \phi_2| \\
	& \geq \xi_1 \big( |\partial_{t_1}\phi_1| - \rho^\prime  |\partial_{t_1} \phi_2| \big).
\end{align*}
Thus, 
\begin{align*}
	|\nabla_t P_\xi |^2 & \geq \xi_1^2 \Big( \big(|\partial_{t_1}\phi_1| - \rho^\prime  |\partial_{t_1} \phi_2| \big)^2 + \big(|\partial_{t_2}\phi_1| - \rho^\prime  |\partial_{t_2} \phi_2| \big)^2 \Big)\\
	& \geq \frac{1}{2} \xi_1^2 \Big( |\partial_{t_1}\phi_1| + |\partial_{t_2}\phi_1| - \rho^\prime \big( |\partial_{t_1} \phi_2| + |\partial_{t_2} \phi_2|\big) \Big)^2 \\
	& =\frac{1}{2} \xi_1^2 \Big(\|\nabla \phi_1(t)\|_1 - \rho^\prime \|\nabla \phi_2(t)\|_1\Big)^2\\
	& \geq \frac{1}{2}\xi_1^2 (1-\rho^\prime/\rho)\|\nabla \phi_1(t)\|_1^2.
\end{align*}
Therefore, the set $\{ t \in \mathbb{R}^n \; : \; |\nabla P_\xi(t)| \leq 1\}$ is contained in the set
\begin{equation*}
	\left\{ t \in \mathbb{R}^2 \; : \; \|\nabla \phi_1(t)\|_1^2 \leq \frac{2\rho}{\xi_1^2(\rho - \rho^\prime)} \right\},
\end{equation*}
 for every $\xi \in \mathbb{S}^{1}$ such that $0 < \xi_2 < \rho^\prime \xi_1$. So assuming that for every constant $C>0$, the set $\{ t \in \mathbb{R}^2 \; : \; \|\nabla \phi_1(t)\|_1 \leq C\}$ is bounded, then the map $\phi$ satisfies condition $(\star)$ with $\theta = 1-\frac{1}{k}$ at $(1,0)$.
\end{exa}

\section*{Acknowledgements}
We would like to thank professors Paulo D. Cordaro and Gerson Petronilho for the insightful discussions. The author was supported by S\~ao Paulo Research Foundation (FAPESP), Grant 2020/15368-7.

\comment{
\section*{Declarations}

\textbf{Data availability} Data sharing not applicable to this article as no datasets were generated or analyzed during the current study.
\quad \\
\quad \\
\noindent\textbf{Conflict of interest} On behalf of all authors, the corresponding author states that there is no conflict of interest.
}

\end{document}